\documentclass{article}

\usepackage{amsmath, amsthm, amssymb, bbm}
\usepackage{xcolor}
\usepackage[colorlinks=true,allcolors=green!50!black]{hyperref}
\usepackage{geometry}
\usepackage{float}
\usepackage{babel}
\usepackage{cases}
\usepackage{enumerate}
\usepackage{mathtools}
\usepackage{algorithm}
\usepackage{algpseudocode}
\usepackage{multirow}
\usepackage{booktabs}
\usepackage{pifont}
\usepackage[output-decimal-marker={.},separate-uncertainty = true]{siunitx}
\usepackage{xr-hyper}
\usepackage{hyperref}
\usepackage{cleveref}
\usepackage{cite}
\newtheorem{theorem}{Theorem}[section]
\newtheorem{lemma}{Lemma}[section]

\newtheorem{proposition}{Proposition}[section]
\theoremstyle{definition}
\newtheorem{assumption}{Assumption}[section]
\theoremstyle{remark}
\newtheorem{remark}{Remark}

\crefname{assumption}{assumption}{assumptions}

\newcommand{\xmark}{\ding{55}}%

\DeclareMathOperator{\diag}{diag}
\DeclareMathOperator{\rank}{rank}

\DeclareMathOperator{\nnz}{nnz}

\newcommand{\tr}{\textrm{tr}}
\newcommand{\tol}{\textrm{tol}}
\newcommand{\ie}{i.e., }
\newcommand{\R}{\mathbb{R}}
\newcommand{\F}{\mathcal{F}}
\newcommand{\I}{\mathcal{I}}
\newcommand{\J}{\mathcal{J}}
\def\inerror{\varrho}
\def\inerrornormeq{\vartheta}

\DeclareMathOperator{\minf}{\mathrm{minimize}}

\DeclareMathOperator{\maxf}{\mathrm{maximize}}
\newcommand{\st}{\mathrm{subject}~\mathrm{to}}

\newcommand{\mnote}[1]{}
\renewcommand{\mnote}[1]{\textcolor{purple}{\textbf{[MU: #1]}}}

\algrenewcommand\algorithmicrequire{\textbf{Input:}}
\algrenewcommand\algorithmicensure{\textbf{Output:}}

\renewcommand{\Comment}[1]{%
\hfill \mbox{$\triangleright$\,#1}}

\makeatletter
\newcommand*{\addFileDependency}[1]{
\typeout{(#1)}
\@addtofilelist{#1}
\IfFileExists{#1}{}{\typeout{No file #1.}}
}\makeatother

\title{Randomized Nystr\"om Preconditioned Interior Point-Proximal Method of Multipliers}

\author{Ya-Chi Chu\footnote{Department of Mathematics, Stanford University, CA, United States; email: ycchu97@stanford.edu} 
\and Luiz-Rafael Santos\footnote{Department of Mathematics, Federal University of Santa Catarina, Blumenau, SC, Brazil; email: l.r.santos@ufsc.br}
\and Madeleine Udell\footnote{Department of Management Science and Engineering, Stanford University, CA, United States; email: udell@stanford.edu} 
}

\begin{document}
\maketitle

\begin{abstract}
We present a new algorithm for convex separable quadratic programming (QP) called Nys-IP-PMM, a regularized interior-point solver that uses low-rank structure to accelerate solution of the Newton system. The algorithm combines the interior point proximal method of multipliers (IP-PMM) with the randomized Nystr\"om preconditioned conjugate gradient method as the inner linear system solver. Our algorithm is matrix-free: it accesses the input matrices solely through matrix-vector products, as opposed to methods involving matrix factorization. It works particularly well for separable QP instances with dense constraint matrices. We establish convergence of Nys-IP-PMM. Numerical experiments demonstrate its superior performance in terms of wallclock time compared to previous matrix-free IPM-based approaches.
\end{abstract}



\section{Introduction} 
Solving a large-scale quadratic optimization problem to high precision, such as \num{e-6} to \num{e-8} relative accuracy, poses a considerable challenge. 
First-order methods scale and parallelize well but converge so slowly that accuracy below around \num{e-4} is generally not achievable.
Conversely, interior point methods achieve high accuracy, 
but are generally only used for problems with $n \leq \num{10000}$ dimensions, as standard implementations require factorization of a $n \times n$ matrix at a cost of $\mathcal{O}(n^3)$.

Matrix-free interior point solvers present a solution to this conundrum:
these solvers access the original input only through matrix-vector products (\emph{matvecs}) \cite{Gondzio_2012_mfIPM},
and so can benefit from algorithmic advances and hardware for fast matrix-vector multiplication.
For example, an $n \times n$ discrete Fourier transform (DFT) matrix can be applied to a vector in $O(n \log n)$ time \cite[sect.~1.4]{golub2013matrix};
and for general dense matrices, hardware accelerators such as GPUs enable fast matvecs. 

Our work aims to solve a primal-dual pair of linearly constrained separable convex quadratic programs (QPs):
\begin{equation*}
\begin{array}{rll}
   \text{(P):} & \underset{x}{\minf}  & \displaystyle\tfrac{1}{2} x^T Q x + c^T x \\
   &\st  & Ax = b,\\
   & & x \geq 0;
\end{array} \quad \quad
\begin{array}{rll}
   \text{(D):} & \underset{x,y,z}{\maxf}  & \displaystyle  b^T y -\tfrac{1}{2} x^T Q x \\
   &\st  & A^T y + z = Qx + c, \\
   & & y \text{: free}, ~ z \geq 0,
\end{array}
\end{equation*}
where $Q \in \R^{n \times n}$ is positive semi-definite diagonal matrix, $A \in \R^{m \times n}$, $c$, $x$, and $z \in \R^n$, and $b$, $y \in \R^m$. 
We assume that $Q$ is diagonal throughout this paper, \ie problem (P) is separable\footnote{
    A non-diagonal $Q$ can be transformed into a diagonal form through a change of variables, but one factorization of $Q$ is required and 
    this transformation may destroy sparsity in the problem.
    Let $Q = U \Lambda U^T$ be an eigendecomposition of $Q$. 
    Define $x' = U^T x$.
    Then the diagonalized objective is $\tfrac{1}{2} x'^T \Lambda x' + c^T x'$ with constraints $A U x' = b$, $U x' \geq 0$.
    Moreover, the inequality constraints $U x' \geq 0$ can be reformulated as linear constraints by introducing slack variables.}. 

This problem template includes many real-world problems.
For example, many control problems, ranging from robotics to aeronautics to finance, model the control effort to be minimized as the sum of squares of control inputs (a diagonal quadratic) subject to given initial and final states and to linear dynamics (linear constraints).

\paragraph{Contributions} 
This work investigates variants of a regularized interior point fra\-me\-work, 
the interior point proximal method of multipliers (IP-PMM) \cite{PougkakiotisGondzio_2021_IPPMM},
for solving separable QPs.
These variants use iterative methods to solve the Newton systems that arise as IP-PMM subproblems.
Regularization is critical for matrix-free methods, which must rely on iterative linear system solvers 
rather than (generally more stable) direct methods.
We propose to use the randomized Nystr\"om preconditioner \cite{Frangella_2023_RandNysPCG}
to accelerate the iterative solve,
and call the resulting algorithm Nys-IP-PMM.
We show that Nys-IP-PMM enjoys both numerical stability and faster convergence than other variants of IP-PMM.
Nys-IP-PMM solves the Newton system inexactly at each IPM iteration, so it is also an inexact variant of matrix-free IP-PMM. 
We prove that, for any $\epsilon \in (0,1)$, inexact IP-PMM achieves duality measure $\mu_k \leq \epsilon$ after $k = O(n^4 \log\frac{1}{\epsilon})$ iterations, provided the error in the search direction decreases in the order of duality measure $\mu_k$ (see \Cref{thm:inexact-IPPMM-QP}). 
This result allows us to establish probabilistic convergence results for Nys-IP-PMM (see \Cref{thm:main}).

In our experiments, we compare the randomized Nystr\"om preconditioner with a more standard choice of preconditioner, the partial Cholesky preconditioner \cite{Gondzio_2012_mfIPM, bellavia2013matrix, morini2018partial}, to assess which one improves the performance of IP-PMM the most. We demonstrate that the randomized Nystr\"om preconditioner generally improves the condition number of the normal equations compared to the partial Cholesky preconditioner.
The results also demonstrate a significant speed-up in terms of wallclock time when using Nys-IP-PMM compared with using partial Cholesky as the preconditioner in IP-PMM, called Chol-IP-PMM.
We provide a publicly available implementation of IP-PMM in Julia 
that incorporates both preconditioners and additional industry-standard heuristics such as Mehrotra's initial point and predictor-corrector method. 
This implementation is the first matrix-free regularized interior point method that is open source and freely available to 
use or modify: see \url{https://github.com/udellgroup/Nys-IP-PMM}. 

\paragraph{Comparison to other variants of IP-PMM}
\Cref{tab:IPPMM} summarizes the differences between previous work on IP-PMM and our contributions.
The first paper to propose IP-PMM \cite{PougkakiotisGondzio_2021_IPPMM} considers QP with exact search direction (\ie using a direct linear system solver), providing polynomial complexity results and numerical experiments.
Numerical performance of inexact IP-PMM on QP has been explored in experimental papers \cite{bergamaschi2021new, gondzio2022general} by the same authors along with other researchers. These papers in\-tro\-duce novel preconditioners for the iterative solvers used by inexact IP-PMM, but they use entrywise access to $Q$ and $A$, and so are not matrix-free.
However, no theoretical convergence result of inexact IP-PMM on QP is presented.
The first theoretical convergence analysis of inexact IP-PMM is established in \cite{2022IPPMM-SDP}, but on linearly constrained semidefinite programming (SDP), without implementation. In our work, we adapt the convergence proof of inexact IP-PMM for SDP \cite{2022IPPMM-SDP} as a theoretical basis for inexact IP-PMM algorithm for solving QP. 
We propose the randomized Nystr\"om PCG, an iterative matrix-free solver, within inexact IP-PMM and establish probabilistic convergence results, along with numerical experiments to show the efficacy of the method in practice.

\begin{table}[]
\centering
\caption{Features for variants of IP-PMM for convex QP / SDP}
\label{tab:IPPMM}
\resizebox{\linewidth}{!}{
\begingroup
\begin{tabular}{rccccc}
\toprule
    & Problem & \shortstack{Search \\ direction} & \shortstack{Convergence \\ proof} & \shortstack{Matrix-free \\ preconditioning} & \shortstack{Numerical \\ results} \\ \cmidrule(l){2-6}
Pougkakiotis and Gondzio \cite{PougkakiotisGondzio_2021_IPPMM}  &  QP &  Exact &  \checkmark   &  -   & \checkmark \\
Bergamaschi et al. \cite{bergamaschi2021new}, Gondzio et al. \cite{gondzio2022general} & QP & Inexact &   \xmark & \xmark  &  \checkmark  \\
Pougkakiotis and Gondzio \cite{2022IPPMM-SDP}    &   SDP  &  Inexact  &   \checkmark  & \xmark & \xmark  \\
Nys-IP-PMM (ours) &  QP  &  Inexact & \checkmark &  \checkmark &   \checkmark    \\ 
\bottomrule
\end{tabular}
\endgroup
}
\end{table}

\paragraph{Organization}
\Cref{sec:background} contains an overview of IP-PMM, randomized Nystr\"om preconditioner, and partial Cholesky preconditioner.
\Cref{sec:inexact-IPPMM} details convergence results for inexact IP-PMM on QP.
\Cref{sec:Nys-IP-PMM} introduces our main algorithm, Nys-IP-PMM, proves convergence of the algorithm, and provides implementation details.
\Cref{sec:numerical-exp} demonstrates the numerical performance of Nys-IP-PMM.

\paragraph{Notation} 
For a vector $x \in \R^n$, subindex $x_{i} \in \R$ denotes the $i$-th component of $x$, and superindex $x^{k} \in \R^n$ denotes the $k$-th iterate. 
For scalar/matrix, we use subscript $k$ for the iteration count, as usual.
Given a set of indices $\mathcal{S} \subset [n] \coloneqq \{1, 2, \ldots, n\}$ and an arbitrary vector $x \in \R^n$, let $x_{\mathcal{S}}$ denote the sub-vector containing the elements of $x$ wholse indices belong to $\mathcal{S}$.
The diagonal matrix with main diagonal $x$ is denoted by $\diag(x)$.
Given a matrix $M$, we use $\lambda_i(M)$ to denote the $i$-th largest eigenvalue of $M$ and use $M^{\dagger}$ to denote the pseudoinverse of $M$.
For a positive definite $M$, we write the $M$-norm $\|u\|_M^2 = u^T M u$.
The set of $m \times m$ real symmetric positive semidefinite matrices is denoted by $\mathbb{S}_+^m(\R)$.
The $n$-by-$n$ identity matrix is denoted by $I_n$ and $\mathbbm{1}_n = (1, \ldots, 1)^T \in \R^n$.
The cost of computing a matrix-vector product with given matrix $M$ is denoted by $T_{M}$.

\section{Background} \label{sec:background}

\subsection{Interior-point methods (IPMs)} 
IPMs are among the most important methods, both in theory and in practice, for solving convex QPs.
IPMs were pioneered by the landmark paper of Karmarkar \cite{karmarkar1984new} in 1984 on linear programming (LP), with a flurry of activity improving these initial results both in theory \cite{megiddo1989pathways, kojima1989primal, kojima1993primal} and in practice \cite{marsten1990interior, Mehrotra_1992_implementation, lustig1992implementing, lustig1994interior}. 
One key reason for the success of the IPM approach is the use of the logarithmic barrier function \cite{gill1986projected}, a nonlinear programming technique.
Later, IPMs were extended to linearly constrained convex QP \cite{kapoor1986fast, ye1989extension, mehrotra1990algorithm, goldfarb1990n, monteiro1989interior},  semidefinite programming (SDP) \cite{vandenberghe1996semidefinite}, and more generally conic optimization problems \cite{nesterov1994interior, forsgren2002interior, lobo1998applications}.
For a more complete history of IPMs, see the review paper \cite{gondzio2012_IPM25}. 

\subsubsection{Regularized IPMs} In primal–dual IPMs for (P) and (D) with diagonal $Q$, the majority of the computation is devoted to solving a symmetric positive semi-definite linear system to determine the search direction, $(\Delta x^{k}, \Delta y^{k}, \Delta z^{k})$, for each iteration $k$.
At each iteration, the algorithm forms a new diagonal matrix $\Theta_k$ and right-hand side (RHS) vector $\xi^{k}$ 
and must solve the \emph{normal equations} 
\begin{equation} \label{eq:normal-eq-nonreg}
A (Q + \Theta_k^{-1})^{-1} A^T \Delta y^{k} = \xi^{k}.
\end{equation}
The system \eqref{eq:normal-eq-nonreg}, whether solved by direct or iterative methods, 
is unstable when $A$ is nearly rank-deficient or when $Q+\Theta_k^{-1}$ is nearly singular.
\textit{Regularized IPMs} \cite{saunders1996cholesky, saunders1996solving, altman1999regularized, friedlander2012primal} improve stability by regularizing the primal problem (P) and/or dual problem (D).
Compared to standard IPMs, regularized IPMs generally require more iterations (linear system solves), but the problem to solve at each iteration is more stable \cite{PougkakiotisGondzio_2021_IPPMM}.
Regularization is particularly important to matrix-free methods, which must rely on 
iterative linear system solvers rather than (more stable) direct methods.

\subsubsection{Interior point proximal method of multipliers (IP-PMM)}
IP-PMM is the first provably polynomial-time primal-dual regularized IPM, proposed by Pougkakiotis and Gondzio \cite{PougkakiotisGondzio_2021_IPPMM} in 2021.
At each IP-PMM iteration, the algorithm determines the search direction by applying a \emph{single} IPM iteration to the PMM subproblem associated with the primal-dual pair (P)--(D) \cite{rockafellar1976augmented}:
\begin{equation} 
    \label{eq:PMM-subproblem}
    \begin{array}{cl}
    \underset{x \geq 0}{\minf} & \displaystyle
    \frac{1}{2} x^T Q x + c^T x + (\lambda^{k})^T (b - Ax) + \frac{1}{2 \delta_k} \|Ax - b\|_2^2 + \frac{\rho_k}{2} \|x - \zeta^{k}\|^2.
    \end{array}
\end{equation}
These PMM subproblems are parametrized by estimates $\zeta^{k}$ and $\lambda^{k}$ of the primal variable $x$ and Lagrange multiplier $y$ respectively, and parameters $\rho_k > 0$ and $\delta_k > 0$ controlling the strength of the regularization.
The linear and quadratic terms $(\lambda^{k})^T (b - Ax) + \frac{1}{2 \delta_k} \|Ax - b\|_2^2$ are motivated by the \textit{method of multipliers} (also called the \textit{augmented Lagrangian method}) \cite{hestenes1969multiplier, powell1969method}, and ensure the dual objective is strongly concave. 
The \emph{proximal} term $\frac{\rho_k}{2} \|x - \zeta^{k}\|^2$ provides strong convexity for primal problem and ensures better numerical behavior. 
When $\rho_k = \delta_k \rightarrow 0$ and the dual estimate is updated by gradient ascent $\lambda^{k+1} \leftarrow \lambda^{k} - \frac{1}{\delta_k}(A x^{k+1} - b)$, the solution to the PMM subproblem \eqref{eq:PMM-subproblem} converges to the solution to (P) \cite{rockafellar1976augmented}.
Note that IP-PMM does not solve the PMM subproblem to any particular precision, 
but simply applies a single IPM iteration to \eqref{eq:PMM-subproblem} to obtain a search direction and update the iterates. 

To apply an IPM iteration to \eqref{eq:PMM-subproblem} and find the search direction, a logarithmic barrier function is introduced to enforce the non-negativity constraint $x \geq 0$, so that the Lagrangian to minimize becomes
\begin{multline} \label{eq:IPPMM-Lagrangian}
\mathcal{L}^{\text{IP-PMM}}_{\mu_k, \rho_k, \delta_k}(x; \zeta^{k}, \lambda^{k}) 
= \frac{1}{2} x^T Q x + c^T x + (\lambda^{k})^T (b - Ax) + \frac{1}{2 \delta_k} \|Ax - b\|_2^2 +  \\ 
+ \frac{\rho_k}{2} \|x - \zeta^{k}\|^2 - \mu_k \sum_{j=1}^n \ln x^{j}.
\end{multline}
Introducing the new variables $y = \lambda^{k} - \frac{1}{\delta_k}(Ax-b)$ and $z = \mu_k \diag(x)^{-1} \mathbbm{1}_n$, the first-order optimality conditions for minimizing \eqref{eq:IPPMM-Lagrangian} result in the nonlinear system
\begin{equation} \label{eq:F=0}
\begin{bmatrix}
c + Qx - A^T y -z + \rho_k (x - \zeta^{k}) \\
Ax + \delta_k(y - \lambda^{k}) - b \\
\diag(x) z - \mu_k \mathbbm{1}_n
\end{bmatrix}
=
\begin{bmatrix}
0 \\ 0 \\ 0
\end{bmatrix}.
\end{equation}
IP-PMM applies a variation of Newton's method to solve a perturbed form of \eqref{eq:F=0}. 
Specifically, it linearizes the optimality conditions and perturbs the right-hand side (RHS) 
(see \cref{sec:IPPMM-tech-def} for details and \cref{eq:inexact-Newton-RHS} for definitions of $r_{d}^{k} \in \R^n$, $r_{p}^{k} \in \R^m,$ and $r_{\mu}^{k} \in \R^n$):
\begin{equation} \label{eq:Newton}
\begin{bmatrix}
-\left(Q+\rho_k I_n\right) & A^T & I_n \\
A & \delta_k I_m & 0 \\
\diag(z^{k}) & 0 & \diag(x^{k})
\end{bmatrix}\begin{bmatrix}
\Delta x^{k} \\
\Delta y^{k} \\
\Delta z^{k}
\end{bmatrix}
=
\begin{bmatrix}
r_{d}^{k} \\
r_{p}^{k} \\
r_{\mu}^{k}
\end{bmatrix}.
\end{equation}
When $Q$ is diagonal, by eliminating $\Delta x^{k}$ and $\Delta z^{k}$ using the first and third block of equations in \eqref{eq:Newton}, the Newton system \eqref{eq:Newton} can be reduced to solving \eqref{eq:normal-eq} and computing $\Delta x^{k}$ and $\Delta z^{k}$ via \eqref{eq:delta_x}--\eqref{eq:delta_z}
\begin{align}
&(N_k + \delta_k I_m) \Delta y^{k} = \xi^{k}; \label{eq:normal-eq} \\
&\Delta x^{k} = (Q + \Theta_k^{-1} + \rho_k I_n)^{-1} (A^T \Delta y^{k} - r_{d}^{k} + \diag(x^{k})^{-1} r_{\mu}^{k}); \label{eq:delta_x} \\
&\Delta z^{k} = \diag(x^{k})^{-1} (r_{\mu}^{k} - \diag(z^{k}) \Delta x^{k}), \label{eq:delta_z}
\end{align}
where $N_k \coloneqq A (Q + \Theta_k^{-1} + \rho_k I)^{-1} A^T$, $\Theta_k = \diag(x^{k}) \diag(z^{k})^{-1}$ is diagonal, and
\begin{equation*}
\xi^{k} = r_{p}^{k} + A (Q + \Theta_k^{-1} + \rho_k I_n)^{-1} (r_{d}^{k} - \diag(x^{k})^{-1} r_{\mu}^{k}).
\end{equation*}
Equation \eqref{eq:normal-eq} is the regularized version of normal equations \eqref{eq:normal-eq-nonreg} in original IPMs and is solved to determine the search direction of IP-PMM.


The polynomial complexity results of IP-PMM are built upon the choice $\rho_k = \delta_k = \mu_k$ and the barrier parameter $\mu_k$ is chosen to be the duality measure $\mu_k := \frac{x_k^T z_k}{n}$, as in standard IPMs.
That is, the convergence theory for IP-PMM requires the duality measure $\mu_k$ to control both the IPM and PMM.
Under this choice and standard assumptions, Pougkakiotis and Gondzio \cite[Thm.~3]{PougkakiotisGondzio_2021_IPPMM} show that every limit point of $\{(x_k, y_k, z_k) : k \in \mathbb{N}\}$ generated by IP-PMM determines a primal-dual solution of pair (P)--(D).
They show linear convergence of the method:
for any given tolerance error $\epsilon \in (0,1)$, IP-PMM produces a sequence of iterates $\{(x_k, y_k, z_k)\}_{k \in \mathbb{N}}$ such that $\mu_k \leq \epsilon$ after $O(n^4 \log \frac{1}{\epsilon})$ iterations \cite[Thm.~2]{PougkakiotisGondzio_2021_IPPMM}.
Their results assume the Newton system \eqref{eq:Newton} is solved exactly.
While this assumption is reasonable when a direct method solves \eqref{eq:normal-eq}, 
it becomes untenable when using an iterative method such as CG. 
We prove convergence for \textit{inexact} IP-PMM for QP in \cref{sec:inexact-IPPMM}.
Note that the resulting algorithm, presented as \Cref{alg:inexact-IPPMM} has two loops: 
the outer loop of IP-PMM and the inner loop of the iterative solver for the normal equations \eqref{eq:normal-eq}.

\subsection{Preconditioned conjugate gradient} The conjugate gradient (CG) me\-thod is the most common iterative method to solve \eqref{eq:normal-eq}.
However, even with the extra stability conferred by IP-PMM, 
CG usually \emph{stalls}, i.e. fails to converge to sufficient accuracy, 
especially during the final iterations of IP-PMM.
Preconditioning is used to improve convergence of CG \cite[sect.~11.5]{golub2013matrix} by 
choosing a positive definite matrix $P$, called a \textit{preconditioner},
replacing the system to be solved with an equivalent symmetric positive definite system: 
\begin{equation*}
P^{-1/2} (N_k + \delta_k I_m) P^{-1/2} v^{k} = P^{-1/2} \xi^{k}
\end{equation*}
with new variable $v^{k} = P^{1/2} \Delta y^{k}$. 
The resulting algorithm is called \textit{preconditioned conjugate gradient} (PCG), and its update relies on the matrix-vector product $v \mapsto P^{-1} v$. 
Therefore, a good preconditioner has two properties: 
1) its inverse $P^{-1}$ must be computationally cheap to apply, and 
2) the preconditioned system $P^{-1/2} (N_k + \delta_k I_m) P^{-1/2}$ must have a spectrum that allows CG to converge faster than on the original system $N_k + \delta_k I_m$. 
For example, we may seek a preconditioner so that the preconditioned system has a reduced condition number \cite{johnson1983polynomial, chan1988optimal, de2017condition}.

\subsubsection{Randomized IPMs for linear programming} Many preconditioners in the IPM literature require entrywise access to the matrices $Q$ and $A$ \cite{bergamaschi2021new, gondzio2022general, Chowdhury_2022_RandIPM, oliveira2005new, bocanegra2007using, durazzi2003indefinitely, schork2020implementation, casacio2017improving}. 
Among these, the randomized preconditioner by Chowdhury et al. \cite{Chowdhury_2022_RandIPM} is most related to our work.
Chowdhury et al. \cite{Chowdhury_2022_RandIPM} pioneered the use of randomized preconditioners in non-regularized IPMs for LPs (\ie $Q = 0$) with wide constraint matrix $m \ll n$.
The same randomized preconditioner is later used to accelerate the predictor-corrector IPM for LPs in a follow-up work \cite{dexter2022convergence}.
In their works, the normal equations take the form $A \Theta_k A^T \Delta y^{k} = \xi^{k}$. 
They propose the preconditioner $P = A\Theta_k^{-1/2} \Omega \Omega^T \Theta_k^{-1/2}A^T$, where $\Omega \in \R^{n \times \ell}$ is a random sketching matrix such that $P$ approximates $A \Theta_k A^T$ well in spectral norm with probability at least $1 - \eta$ \cite[Lemma 2]{Chowdhury_2022_RandIPM}.
The sketch size $\ell$ is of order $\ell = O(m \log(m / \eta))$ and matrix $\Omega$ has $O(\log(m / \eta))$ non-zero entries.
As a result, the total computational cost for $P^{-1}$ is $O( (\nnz(A) + m^3) \log(m / \eta))$.
The method suffers cubic complexity in the number of constraints $m$. In contrast, 
the Nystr\"om preconditioner can be constructed in $O(\ell \nnz(A) + \ell^3)$ and thus
offers the greatest advantage over \cite{Chowdhury_2022_RandIPM} 
when $m$ is much larger than the rank $\ell$ required for a good approximation.

\subsubsection{Partial Cholesky preconditioners} A matrix-free preconditioner offers improved scalability,
particularly when entrywise access to $A$ is expensive.
Gondzio \cite{Gondzio_2012_mfIPM} propose a matrix-free preconditioner in the context of (regularized) IPMs, called the \textit{partial Cholesky preconditioner}, which is elaborated in Bellavia et al. \cite{bellavia2013matrix} and Morini \cite{morini2018partial}.
Given a target rank $\ell \ll m$, 
partial Cholesky forms the greedily pivoted Cholesky factorization $LL^T = E(N_k + \delta_k I)E^T$ 
with appropriate permutation matrix $E \in \R^{m \times m}$ 
and takes the first $\ell$ columns of $L$, $L_{\ell} \in \R^{m \times \ell}$, 
which corresponds to a rank-$\ell$ approximation $L_{\ell} L_{\ell}^T \approx N_k + \delta_k I$,
together with a diagonal matrix that ensures the approximation is positive definite. 
The key limitation of the partial Cholesky preconditioner lies in its construction, which relies on the entire diagonal of $N_k$. 
In a matrix-free setting without entrywise access to $A$, 
computing the diagonal of $N_k$ requires at least $m$ matvecs with $A^T$.
In practice, this feature slows down the partial Cholesky preconditioner considerably, 
as these additional $m$ matvecs are performed at every iteration.

We describe the partial Cholesky preconditioner in more detail here for com\-plete\-ness, 
since it is a standard choice for a matrix-free IPM implementation.
We omit the iteration count $k$ when it is clear from context.
The first $\ell$ columns of the Cholesky factor, $L_{\ell} \coloneqq \begin{bmatrix}
L_{11} \\
L_{21}
\end{bmatrix}$, yield the following factorization in block form:
\begin{equation} \label{eq:partialCholeskyfact}
E(N + \delta I)E^T 
\coloneqq
\begin{bmatrix}
N_{\delta, 11} & N_{\delta, 21}^T \\
N_{\delta, 21} & N_{\delta, 22}
\end{bmatrix}
= 
\begin{bmatrix}
L_{11} & \mathbf{0} \\
L_{21} & I_{m-\ell}
\end{bmatrix}
\begin{bmatrix}
I_\ell & \mathbf{0} \\ \mathbf{0} &  S
\end{bmatrix}
\begin{bmatrix}
L_{11}^T & L_{21}^T \\
\mathbf{0} & I_{m-\ell}
\end{bmatrix}, 
\end{equation}
where $N_{\delta, 11} \in \R^{\ell \times \ell}$ is the leading principal submatrix containing the $\ell$ largest diagonal elements and $S = N_{\delta, 22} - N_{\delta, 21} N_{\delta, 11}^{-1} N_{\delta, 21}^T$ is the Schur complement of $N_{\delta, 11}$.
In practice, $S$ is never explicitly formed when performing the Cholesky factorization and hence is not available.
The partial Cholesky preconditioner approximates $S$ by its diagonal and takes the form
\begin{equation} \label{eq:chol-precond}
P_{\text{Chol}} = E^T
\begin{bmatrix}
L_{11} & \mathbf{0} \\
L_{21} & I_{m-\ell}
\end{bmatrix}
\begin{bmatrix}
I_{\ell} & \mathbf{0} \\ \mathbf{0} &  \diag(S)
\end{bmatrix}
\begin{bmatrix}
L_{11}^T & L_{21}^T \\
\mathbf{0} & I_{m-\ell}
\end{bmatrix} E, 
\end{equation}
the inverse of which has the closed-form formula
\begin{equation} \label{eq:chol-precond-inv}
P_{\text{Chol}}^{-1} = E^T 
\begin{bmatrix}
L_{11}^{-T} & -L_{11}^{-T} L_{21}^T \\
\mathbf{0} & I_{m-\ell}
\end{bmatrix}
\begin{bmatrix}
I_{\ell} & \mathbf{0} \\ \mathbf{0} &  \diag(S)^{-1}
\end{bmatrix}
\begin{bmatrix}
L_{11}^{-1} & \mathbf{0} \\
-L_{21} L_{11}^{-1} & I_{m-\ell}
\end{bmatrix} E. 
\end{equation}
Pseudocode for the construction of this preconditioner 
appears in \cite[Alg.~1]{bellavia2013matrix} and \cite[Alg.~1--2]{morini2018partial}, 
while spectral analysis for the preconditioned system is developed in \cite{bellavia2013matrix, Gondzio_2012_mfIPM}.
In our comparison, we construct the partial Cholesky preconditioner by \cite[Alg.~2]{morini2018partial},
which yields a more efficient matvec with $P_{\text{Chol}}^{-1}$ when the first $\ell$ columns of $E(N + \delta I)E^T$ are more sparse than $L_{11}^{-1}$ and $L_{21}$; and a comparable cost of matvec otherwise \cite[sect.~3.1]{morini2018partial}.

\begin{remark}[sparsity of $L_{21}$]
In this paper, we consider both sparse and dense constraint matrices $A$. 
When $A$ is dense, the Cholesky factors are in general dense.
When the $L_{21}$ block of the Cholesky factorization is sparse, 
the matvec with the partial Cholesky preconditioner can potentially exploit the sparsity of $L_{21}$.
However, even when $A$ is sparse, the $L_{21}$ block of the Cholesky factorization may not be sparse.
Sparsity of this block depends on the sparsity pattern of $A$ and the pivots used.
However, it is difficult to choose these pivots well without entrywise access to $A$.
In the partial Cholesky literature, these pivots are chosen greedily based on the diagonal entries of $N_k$.
In the problems we consider in our numerical experiments, the $L_{21}$ block is always dense.
\end{remark}

Constructing the partial Cholesky preconditioner requires access to the complete diagonal of $N + \delta I$ since we need entry-wise access of $N_{\delta, 11}$ for factorization \eqref{eq:partialCholeskyfact} and $\diag(N_{\delta, 22})$ for computing $\diag(S)$.
Using the form $N = A (Q + \Theta^{-1} + \rho I)^{-1} A^T$, the diagonal of $N + \delta I$ can be computed in the following matrix-free manner:
\begin{equation*}
r_i = (Q + \Theta^{-1} + \rho I)^{-\frac{1}{2}} A^T e_i, \quad (N + \delta I)_{ii} = r_i^T r_i + \delta, \quad \text{for all } i = 1, 2, \ldots, m.
\end{equation*}
This computation requires $m$ matrix-vector products with $A^T$ and componentwise scaling of a length $n$ vector, which costs $O(m (T_A + n))$\footnote{This cost can be reduced to $\mathcal{O}(\nnz(A))$ if entrywise access to $A$ is available, since $A^T e_i$ is the $i$-th row of $A$.\label{fn:entrywise}}.
The diagonal of $S$ can be similarly computed, as a cost $O(m(T_{N_{\delta, 21}} + \ell^2))$.
The total construction cost of partial Cholesky preconditioner is dominated by computing $\diag(N + \delta I)$ and $\diag(S)$, resulting in $O(m (T_{A} + n) + m (T_{N_{\delta, 21}} + \ell^2))$ arithmetic operations.

\subsubsection{Randomized Nystr\"om preconditioners} 
The randomized Nystr\"om pre\-con\-di\-tion\-er is built upon the randomized Nystr\"om low-rank appro\-xi\-ma\-tion, which we introduce now.
Let $N \in \mathbb{S}_+^m(\R)$ be a real symmetric positive semidefinite matrix and $\Omega \in \mathbb{R}^{m \times \ell}$ be a random Gaussian test matrix (\ie each entry is drawn  i.i.d. from a normal distribution). The integer $\ell \ll m$ is called the \textit{sketch size} and the matrix $Y = N \Omega \in \R^{m \times \ell}$ is called the \textit{sketch} of $N$. We observe the sketch can be obtained by $\ell$ matrix-vector products with $N$.
The \textit{randomized Nystr\"om approximation} with respect to the range of $\Omega$ is defined as
\begin{equation} \label{eq:Nystrom-approx}
\hat{N}= (N \Omega) (\Omega^T N \Omega)^{\dagger} (N \Omega)^T = Y (\Omega^T Y)^{\dagger} Y^T.
\end{equation}
Observe from \eqref{eq:Nystrom-approx} that this approximation can be constructed directly from the test matrix $\Omega$ and the sketch $Y$, without further access to $N$. 
The rank of $\hat{N}$ is equal to $\ell$ with probability 1, and hence the terms sketch size and rank are used interchangeably.
The formula \eqref{eq:Nystrom-approx} is not numerically stable.
Instead, our algorithm uses \Cref{alg:nys-approx} to construct a randomized rank-$\ell$ Nystr\"om approximation. It provides a stable and efficient implementation with a computational cost of $O(T_{N} \ell + \ell^2 m)$, where $T_N$ is the cost of a matvec with $N$. 
In pracitce, the thin SVD in line \ref{line:svd} of \Cref{alg:nys-approx} is the most computationally intensive operation.
Moreover, \Cref{alg:nys-approx} returns the truncated eigendecomposition of the randomized Nystr\"om approximation: $\hat{N} = \hat{U} \hat{\Lambda} \hat{U}^T$, where $\hat{U} \in \R^{m \times \ell}$ has orthonormal columns and $\hat{\Lambda} \in \R^{\ell \times \ell}$ is diagonal with diagonal entries $\hat{\lambda}_1 \geq \cdots \geq \hat{\lambda}_{\ell}$. 

\begin{algorithm}
    \caption{Randomized Nystr\"om Approximation \cite[Algorithm 16]{martinsson2020randomized}} 
    \label{alg:nys-approx}
    \begin{algorithmic}[1] 
    \Require{Matrix-vector product oracle of an $m \times m$ positive semidefinite matrix $N$, target rank $\ell \ll m$}
    \Ensure{Randomized Nystr\"om approximation $\hat N = \hat{U} \hat{\Lambda} \hat{U}^T$}

    \State Draw a Gaussian test matrix $\Omega \in \mathbb{R}^{m \times \ell}$
        \Comment{$\mathcal{O}(\ell m)$}
    \State $Y = N \Omega$
        \Comment{$\mathcal{O}(\ell T_N)$}
    \State $\nu=\operatorname{eps}(\operatorname{norm}(Y, \text{`fro'}))$
        \Comment{$\mathcal{O}(\ell m)$}
        
    \State $Y_\nu=Y+\nu \Omega$
        \Comment{$\mathcal{O}(\ell m)$}

    \State $C=\operatorname{chol}\left(\Omega^T Y_\nu\right)$
        \Comment{$\mathcal{O}(\ell^2 m)$}

    \State $B=Y_\nu C^{-1}$
        \Comment{$\mathcal{O}(\ell^2 m)$}

    \State\label{line:svd} $[\hat{U}, \Sigma, \sim]=\operatorname{svd}(B)$
        \Comment{$\mathcal{O}(\ell^2 m)$}

    \State $\hat{\Lambda}=\max \left\{0, \Sigma^2-\nu I\right\}$
        \Comment{$\mathcal{O}(\ell)$}

    \end{algorithmic}
\end{algorithm}

\begin{algorithm}
    \caption{Nystr\"om Preconditioner \cite{Frangella_2023_RandNysPCG}} \label{alg:Nys-precond}
    \begin{algorithmic}[1] 
    \Require{Positive semidefinite $N_k$, sketch size $\ell_k$, regularization parameter $\delta_k$}
    \Ensure{Inverse of randomized Nystr\"om preconditioner $P_k^{-1}$ as a linear operator}
    
    \vspace{0.5pc}
        \State Compute randomized Nystr\"om rank-$\ell_k$ approximation by \Cref{alg:nys-approx}:
        \begin{equation*}
            \hat{N}_k = \hat{U} \hat{\Lambda} \hat{U}^T
        \end{equation*}
        
        \State Construct inverse Nystr\"om preconditioner $P_k^{-1}$ as in \eqref{eq:Nys_precond_inv} with $\hat{U}$, $\hat{\Lambda}$, and $\delta_k$.
            
    \end{algorithmic}
\end{algorithm}

Given a rank-$\ell$ randomized Nystr\"om approximation $\hat{N} = \hat{U} \hat{\Lambda} \hat{U}^T$ returned from \Cref{alg:nys-approx}, the \textit{randomized Nystr\"om preconditioner} for regularized system \eqref{eq:normal-eq} and its inverse take the form
\begin{align}
  P_{\text{Nys}} &= \frac{1}{\hat{\lambda}_{\ell} + \delta} \hat{U} (\hat{\Lambda} + \delta I) \hat{U}^T + (I - \hat{U} \hat{U}^T), \label{eq:Nys_precond} \\[0.5em]
  P_{\text{Nys}}^{-1} &= (\hat{\lambda}_{\ell} + \delta) \hat{U} (\hat{\Lambda} + \delta I)^{-1} \hat{U}^T + (I - \hat{U} \hat{U}^T). \label{eq:Nys_precond_inv}
\end{align}
It is important to highlight that the $m^2$ matrix elements of the randomized Nystr\"om preconditioner are not explicitly formed in practice. Instead, $P_{\text{Nys}}$ is viewed as a linear operator defined by $\hat{U}$, $\hat{\Lambda}$, and $\delta$. 
\Cref{alg:Nys-precond} summarizes the procedure.

In other words, the randomized Nystr\"om preconditioner is directly available once the rank-$\ell$ randomized Nystr\"om approximation has been constructed, and thus shares the same construction cost $O(T_{N} \ell + \ell^2 m)$ as the Nystr\"om approximation.
Both the Nystr\"om preconditioner and its inverse are cheap to apply: a matvec with either 
requires $O(m \ell)$ arithmetic operations, dominated by the cost of applying $\hat{U}$ and $\hat{U}^T$ to a vector.
The required storage for the randomized Nystr\"om preconditioner is $O(m \ell)$.
All these properties offer potential benefits compared to partial Cholesky, particularly when the target rank $\ell$ is much smaller than $m$; see \Cref{tab:comparison}.
Our numerical experiments in \cref{sec:numerical-exp} further demonstrate the advantages of the Nyström preconditioner in terms of both performance and computational efficiency in large-scale (dense) problems.

The randomized Nystr\"om preconditioner can effectively accelerate the con\-ver\-gence of CG on linear systems arising from data-driven models \cite{Frangella_2023_RandNysPCG, zhao2022nysadmm}.
The reason is that most data matrices usually have rapidly decaying spectra and hence have smaller \textit{effective dimension} \cite{Udell_2019_datalowrank,Frangella_2023_RandNysPCG, zhao2022nysadmm, lacotte2020effective}. 
Given a positive semidefinite matrix $N \in \mathbb{S}_+^m(\R)$ and a regularization parameter $\delta > 0$, the \textit{effective dimension} of $N$ is defined as
\begin{equation} \label{eq:effective-dim}
d_{\text{eff}}(N, \delta) = \tr(N (N + \delta I)^{-1}) = \sum_{i=1}^m \frac{\lambda_i}{\lambda_i + \delta},
\end{equation}
where $\lambda_i$'s are eigenvalues of $N$. 
The ratio $\lambda_i/(\lambda_i + \delta)$ is close to $1$ if $\lambda_i \gg \delta$; and is close to $0$ if $\lambda_i \ll \delta$.
As a result, the effective dimension can be understood as a smoothed count of the eigenvalues of $N$ that are greater than or equal to $\delta$.
Zhao et al. \cite{zhao2022nysadmm} show that if the sketch size $\ell = O(d_{\text{eff}}(N, \delta) + \log(\frac{1}{\eta}))$, then with probability at least $1 - \eta$, the condition number of the Nystr\"om preconditioned system is bounded by a constant \cite[Thm.~4.1]{zhao2022nysadmm}, and hence the Nystr\"om preconditioned CG (NysPCG) can achieve $\epsilon$-relative error in $O\left(\log(\frac{1}{\epsilon})\right)$ iterations, independent of the condition number (see \Cref{lem:NysPCG-relative-error} in appendix).

\begin{table}
    \centering
    \caption{Comparison of randomized Nystr\"om preconditioner and the partial Cholesky preconditioner. The construction cost assumes the matrix of the system takes the form $A (Q + \Theta_k + \rho_k I)^{-1} A^T + \delta_k I$ and no entrywise access to $A$ is available.}
    \label{tab:comparison}
    \resizebox{\linewidth}{!}{
    \begingroup
    \begin{tabular}{rccc}
    \toprule
    \textbf{Preconditioner}                   & \textbf{Construction Cost}              & \textbf{Matvec Cost}
    & \textbf{Storage}     \\ \cmidrule(r){1-1} \cmidrule(l){2-4}
    Randomized Nystr\"om & $O(\ell (T_A + n) + \ell^2 m)$ & $2 \ell m + m + 5 \ell$     & $O(\ell m + \ell)$ \\
    Partial Cholesky         & $O(m (T_A + n) + m (T_{N_{\delta, 21}} + \ell^2))$    & $2 \ell m + 2m + 2\ell^2$     & $O(\ell m + \ell^2 + m)$ \\ \bottomrule
    \end{tabular}
    \endgroup
    }
\end{table}

\section{Inexact IP-PMM for convex QP} 
\label{sec:inexact-IPPMM}

Pougkakiotis and Gondzio \cite{PougkakiotisGondzio_2021_IPPMM} assume the search direction in IP-PMM satisfies the Newton system \eqref{eq:Newton} exactly. 
This assumption is unrealistic when the inner linear system solver in IP-PMM is iterative. 
Here, we prove convergence for IP-PMM with errors in the search direction,
which we call \textit{inexact IP-PMM} (see \Cref{alg:inexact-IPPMM}).
Inexact IP-PMM on QP is exactly the same as \cite[Alg.~IP-PMM]{PougkakiotisGondzio_2021_IPPMM}, except that in line 7 of \Cref{alg:inexact-IPPMM}, 
the inexact version we analyze satisfies the inexact Newton system (see \cref{eq:inexact-Newton-RHS} for explicit formula) instead of the exact one.
We adopt the assumption $\rho_k = \delta_k = \mu_k$ \cite{PougkakiotisGondzio_2021_IPPMM} throughout this section.
\Cref{sec:IPPMM-tech-def} introduces some technical definitions, 
and \cref{sec:IPPMM-convergence} presents convergence results for inexact IP-PMM.

\begin{algorithm}[!h]
\caption{Inexact IP-PMM (adapted from \cite[Alg.~IP-PMM]{PougkakiotisGondzio_2021_IPPMM})}
\label{alg:inexact-IPPMM}
\begin{algorithmic}[1] 
\Require{QP data $A, Q, b, c$ as in (P), $\tol$}
\algrenewcommand\algorithmicensure{\textbf{Parameters:}}
\Ensure{$0<\sigma_{\min } \leq \sigma_{\max } \leq 0.5,~ C_N>0, ~0<\gamma_{\mathcal{A}}<1, ~0<\gamma_\mu<1$}
\algrenewcommand\algorithmicensure{\textbf{Starting point:}}
\Ensure{Set as in \eqref{eq:starting-pt}}
\vspace{0.5pc}

\State Compute infeasibility $r_p^{0} \leftarrow A x^{0}-b$ and $r_d^{0} \leftarrow c+Q x^{0}-A^T y^{0}-z^{0}$.

\For{$(k=0,1,2, \cdots)$}
    \vspace{0.2pc}
    \If{$\left(\| r_p^{k} \|_2 < \tol \right) \wedge \left(\| r_d^{k} \|_2 < \tol \right) \wedge \left(\mu_k <  \tol \right)$} 
        \vspace{0.3pc}
        \State \Return $\left(x^{k}, y^{k}, z^{k}\right)$
    \Else
        \State Choose $\sigma_k \in\left[\sigma_{\min }, \sigma_{\max }\right]$. 
        \State Solve system \eqref{eq:inexact-Newton-RHS} such that the inexact errors satisfy \Cref{assump:Newton-error-bound}.
        \State Choose largest stepsize $\alpha_k \in (0,1]$ such that $\mu_k(\alpha) \leq(1-0.01 \alpha) \mu_k$ and
        \begin{equation*}
        {\small
        \begin{bmatrix}
            x^{k}+\alpha_k \Delta x^{k} \\ 
            y^{k}+\alpha_k \Delta y^{k} \\ 
            z^{k}+\alpha_k \Delta z^{k}
        \end{bmatrix}}
        \in \mathcal{N}_{\mu_k(\alpha)}\left(\zeta^{k}, \lambda^{k}\right),
        \end{equation*}
        \hspace{\algorithmicindent*2} where $\mu_k(\alpha) = \left(x^{k}+\alpha_k \Delta x^{k}\right)^T\left(z^{k}+\alpha_k \Delta z^{k}\right) / n$.
        
        \State Update the iterate
        \begin{align*}
        {\small
        \begin{bmatrix}
            x^{k+1} \\ 
            y^{k+1} \\ 
            z^{k+1}
        \end{bmatrix} \leftarrow 
        \begin{bmatrix}
            x^{k}+\alpha_k \Delta x^{k} \\ 
            y^{k}+\alpha_k \Delta y^{k} \\ 
            z^{k}+\alpha_k \Delta z^{k}
        \end{bmatrix}}
        \text{ and } 
        \mu_{k+1} \leftarrow \frac{(x^{k+1})^T z^{k+1}}{n}.
        \end{align*}

        \State Set $r_p^{k+1} \leftarrow A x^{k+1}-b$ and $r_d^{k+1} \leftarrow c+Q x^{k+1}-A^T y^{k+1}-z^{k+1}$.        
        \vspace{0.3pc}

        \State Set $\tilde{r}_p \leftarrow r_p^{k+1} - \frac{\mu_{k+1}}{\mu_0} \bar{b}$ and $\tilde{r}_d \leftarrow r_d^{k+1} + \frac{\mu_{k+1}}{\mu_0} \bar{c}$.
        \vspace{0.3pc}

        \If{$\left( ( \left\|\left(\tilde{r}_p, \tilde{r}_d\right)\right\|_2 \leq C_N \frac{\mu_{k+1}}{\mu_0} ) \wedge (\left\|\left(\tilde{r}_p, \tilde{r}_d\right)\right\|_{\mathcal{A}} \leq \gamma_{\mathcal{A}} \rho \frac{\mu_{k+1}}{\mu_0} ) \right)$}
        \vspace{0.3pc}
        \State $(\zeta^{k+1}, \lambda^{k+1}) \leftarrow (x^{k+1}, y^{k+1})$
        \Else
            \State $(\zeta^{k+1}, \lambda^{k+1}) \leftarrow (\zeta^{k}, \lambda^{k})$
        \EndIf
    \EndIf
    \State $k \leftarrow k+1$
\EndFor
\end{algorithmic}
\end{algorithm}

\subsection{Definitions for inexact IP-PMM} \label{sec:IPPMM-tech-def}
In this section, we provide a detail on inexact IP-PMM,
including the choice of starting point, neighborhoods, and inexact Newton systems, following \cite{PougkakiotisGondzio_2021_IPPMM}.
Refer to \cite{PougkakiotisGondzio_2021_IPPMM} for more details.

\paragraph{Starting point} For the analysis, we consider starting point $(x^{0}, z^{0}) = \rho (\mathbbm{1}_n, \mathbbm{1}_n)$ for some $\rho > 0$ and $y^{0}$ being an arbitrary vector such that $\|y^{0}\|_{\infty} = O(1)$ (i.e., the absolute value of its entries are independent of $n$ and $m$). The initial primal and dual estimates are taken as $\zeta^{0} = x^{0}$ and $\lambda^{0} = y^{0}$, and the initial duality measure is denoted by $\mu_0 = \frac{(x^{0})^T z^{0}}{n}$. 
In addition, there exist two vectors $\bar{b}$ and $\bar{c}$ such that 
\begin{equation} \label{eq:starting-pt}
A x^{0} = b + \bar{b}, ~ -Qx^{0} + A^T y^{0} + z^{0} = c + \bar{c}.
\end{equation}

\paragraph{Neighborhoods}
IP-PMM is a path-following method that requires each iterate of IP-PMM to lie within a specified neighborhood.
The neighborhoods in IP-PMM depend on the parameters in the PMM subproblems \eqref{eq:PMM-subproblem}, as well as the starting point,
and are parametrized by $(\zeta^{k}, \lambda^{k}, \mu_k)$.
IP-PMM uses a semi-norm that depends on the input matrices $A$ and $Q$
to measure primal-dual infeasibility:
\begin{equation*}
\|(b, c)\|_{\mathcal{A}}:=\min _{x, y, z}\left\{\|(x, z)\|_2: A x=b, \; -Q x+A^T y+z=c\right\}.
\end{equation*}
This norm measures the minimum $2$-norm of $(x, z)$ among all primal-dual feasible tuples $(x, y, z)$
and can be evaluated via QR factorization of $A$ 
\cite[sect.~4]{mizuno1999global}. 
Now, given the starting point $(x^{0}, y^{0}, z^{0})$ and vectors $(\bar{b}, \bar{c})$ in \eqref{eq:starting-pt}, penalty parameter $\mu_k$, and primal-dual estimates $\zeta^{k}, \lambda^{k}$, Pougkakiotis and Gondzio \cite{PougkakiotisGondzio_2021_IPPMM} define the set
{\small
\begin{equation} \label{eq:C-tilde-set}
\tilde{\mathcal{C}}_{\mu_k}\left(\zeta^{k}, \lambda^{k}\right):=\left\{
(x, y, z) :
\begin{array}{c}
A x+\mu_k\left(y-\lambda^{k}\right)=b+\frac{\mu_k}{\mu_0}\left(\bar{b}+\tilde{b}^{k}\right), \\
-Q x+A^T y+z-\mu_k\left(x-\zeta^{k}\right)=c+\frac{\mu_k}{\mu_0}\left(\bar{c}+\tilde{c}^{k}\right), \\
\left\|\left(\tilde{b}^{k}, \tilde{c}^{k}\right)\right\|_2 \leq C_N, ~\left\|\left(\tilde{b}^{k}, \tilde{c}^{k}\right)\right\|_{\mathcal{A}} \leq \gamma_{\mathcal{A}} \rho
\end{array} \right\},
\end{equation}
}
where $C_N > 0$ is a constant, $\gamma_{\mathcal{A}} \in (0,1)$, and $\rho > 0$ is defined as in the starting point. The vectors $\tilde{b}^{k}$ and $\tilde{c}^{k}$ represent the current scaled (by $\frac{\mu_0}{\mu_k}$) infeasibility: 
\begin{align*}
\tilde{b}^{k} &\coloneqq \tfrac{\mu_0}{\mu_k} ( A x+\mu_k (y-\lambda^{k} )-b-\tfrac{\mu_k}{\mu_0} \bar{b} ), \\ 
\tilde{c}^{k} &\coloneqq \tfrac{\mu_0}{\mu_k} ( -Q x+A^T y+z-\mu_k (x-\zeta^{k} )-c-\tfrac{\mu_k}{\mu_0} \bar{c} ).
\end{align*}
The set in \eqref{eq:C-tilde-set} contains all points $(x, y, z)$ 
whose scaled infeasibilities are bounded by a constant in both $\|\cdot\|_2$ and $\|\cdot\|_{\mathcal{A}}$.
Then the family of neighborhoods is
\begin{equation} \label{eq:neighborhood}
\mathcal{N}_{\mu_k}\left(\zeta^{k}, \lambda^{k}\right):=\left\{(x, y, z) \in \tilde{\mathcal{C}}_{\mu_k}\left(\zeta^{k}, \lambda^{k}\right) : 
\begin{array}{l}
(x, z) > (0,0), \\[1pt]
x_i z_i \geq \gamma_\mu \mu_k \text{ for all } i
\end{array}
\right\},
\end{equation}
where $\gamma_\mu \in (0,1)$ is a constant that prevents the component-wise complementarity products from approaching zero faster than $\mu_k=(x^{k})^T z^{k} / n$.

\paragraph{Inexact Newton system}
The search direction $(\Delta x^{k}, \Delta y^{k}, \Delta z^{k})$ in the inexact IP-PMM solves the following inexact Newton system:
\begin{multline} \label{eq:inexact-Newton-RHS}
{\left[\begin{array}{ccc}
-\left(Q+\mu_k I_n\right) & A^T & I \\
A & \mu_k I_m & 0 \\
\diag(z^{k}) & 0 & \diag(x^{k})
\end{array}\right]\left[\begin{array}{c}
\Delta x^{k} \\
\Delta y^{k} \\
\Delta z^{k}
\end{array}\right]} \\
=-\left[\begin{array}{c}
-\left(c+\frac{\sigma_k \mu_k}{\mu_0} \bar{c}\right)-Q x^{k}+A^T y^{k}+z^{k}-\sigma_k \mu_k\left(x^{k}-\zeta^{k}\right) \\
A x^{k}+\sigma_k \mu_k\left(y^{k}-\lambda^{k}\right)-\left(b+\frac{\sigma_k \mu_k}{\mu_0} \bar{b}\right) \\
\diag(x^{k}) \diag(z^{k})\mathbbm{1}_n -\sigma_k \mu_k \mathbbm{1}_n
\end{array}\right] 
+ 
\begin{bmatrix}
\inerror_{d}^{k} \\ \inerror_{p}^{k} \\\inerror_{\mu}^{k}
\end{bmatrix}, 
\end{multline}
where the \textit{centering parameter} $\sigma_k \in (0, 1)$ is chosen 
to control how fast the duality measure $\mu_k$ must decrease at the next iteration, and $(\inerror_{d}^{k}, \inerror_{p}^{k}, \inerror_{\mu}^{k})$ model the residuals when the system is solved by iterative methods.

\subsection{Convergence results of inexact IP-PMM on QP} \label{sec:IPPMM-convergence}

This section presents convergence results for inexact IP-PMM on QP, a necessary building block for Nys-IP-PMM. 
We need the following assumption on the residuals in Newton system \eqref{eq:inexact-Newton-RHS}.

\begin{assumption}
\label{assump:Newton-error-bound} 
The residuals $(\inerror_{p}^{k}, \inerror_{d}^{k}, \inerror_{\mu}^{k})$ in inexact Newton system \eqref{eq:inexact-Newton-RHS} satisfy
\begin{equation*}
\inerror_{\mu}^{k}=0, \quad\left\|\left(\inerror_{p}^{k}, \inerror_{d}^{k}\right)\right\|_2 \leq \frac{\sigma_{\min }}{4 \mu_0} C_N \mu_k, \quad\left\|\left(\inerror_{p}^{k}, \inerror_{d}^{k}\right)\right\|_{\mathcal{A}} \leq \frac{\sigma_{\min }}{4 \mu_0} \gamma_{\mathcal{A}} \rho \mu_k,
\end{equation*}
where $C_N, \gamma_{\mathcal{A}}$ are constants defined as in \eqref{eq:C-tilde-set}, $\sigma_{\min}$ is the lower bound for $\sigma_k$ in \Cref{alg:inexact-IPPMM}, and $\rho$ is determined by the starting point chosen in \cref{sec:IPPMM-tech-def}.
\end{assumption}
The assumption $\inerror_{\mu}^{k} = 0$ is reasonable since 
a practical solution method reduces the Newton system \eqref{eq:inexact-Newton-RHS} to the augmented system or the normal equations, and recovers
$\Delta z^{k}$ using a closed-form formula.
The other two assumptions on the inexact errors require that the 
errors are small whenever the iterate is close to a solution.

We also make the following two standard assumptions to ensure the solution 
and the problem data are bounded,
as in \cite{PougkakiotisGondzio_2021_IPPMM,santos2019optimized}.

\begin{assumption} \label{assump:bdd-opt-sol}
The primal-dual QP pair in (P)--(D) 
has an optimal solution $(x^*, y^*, z^*)$ 
with $\left\|x^*\right\|_{\infty} \leq C^*,\left\|y^*\right\|_{\infty} \leq C^*$ and $\left\|z^*\right\|_{\infty} \leq C^*$, for a constant $C^* \geq 0$ independent of $n$ and $m$.
\end{assumption}

\begin{assumption} \label{assump:constraint-mat-full-rank}
Let $\eta_{\min}(A)$ (resp. $\eta_{\max}(A)$) denote the minimum (resp. ma\-xi\-mum) singular value of $A$ and $\nu_{\max}(Q)$ denote the maximum eigenvalue of $Q$.
The constraint matrix of $(\mathrm{P})$ has full row rank $\rank(A)=m$,
and there exist constants $C_{\min}>0$, $C_{\max}>0$, $C_{Q} > 0$, and $C_r>0$, independent of $n$ and $m$, such that $\eta_{\min }(A) \geq C_{\min}$, $\eta_{\max }(A) \leq C_{\max}$, $\nu_{\max}(Q) \leq C_{Q}$, and $\|(c, b)\|_{\infty} \leq C_r$.
\end{assumption}

\Cref{thm:inexact-IPPMM-QP} below provides a polynomial complexity result for inexact IP-PMM, which extends \cite[Thm.~2]{PougkakiotisGondzio_2021_IPPMM}. 
The critical analysis lies in \Cref{lem:bdd-step,lem:IPPMM-stepsize} within \Cref{sec:lemmas}, which guarantee the existence of a stepsize $\bar{\alpha} = O(\frac{1}{n^4})$ at each iteration of inexact IP-PMM.

\begin{theorem} \label{thm:inexact-IPPMM-QP}
Let $\epsilon \in (0,1)$ be a given error tolerance.
Choose a starting point as in \eqref{eq:starting-pt} for inexact IP-PMM and let $C$ and $\omega$ be positive constants such that $\mu_0 \leq \frac{C}{\epsilon^{\omega}}$. 
Assume that at each iteration, the residuals in \eqref{eq:inexact-Newton-RHS} satisfy \Cref{assump:Newton-error-bound}.
Given \Cref{assump:bdd-opt-sol,assump:constraint-mat-full-rank}, the iterates $\{(x^{k}, y^{k}, z^{k})\}_{k \geq 0}$ generated by inexact IP-PMM satisfy $\mu_k \leq \epsilon$ after $k := \frac{n^4}{0.01 \bar{\kappa}}\left[ (1 + \omega) \log(1/\epsilon) + \log C \right] = O\left( n^4 \log \frac{1}{\epsilon} \right)$ iterations, where $\bar{\kappa}$ is a constant independent of $n$ and $m$ defined in \Cref{lem:IPPMM-stepsize}(\ref{item:lem-stepsize-b}).
\end{theorem}
\begin{proof}
Given \Cref{lem:IPPMM-stepsize} that guarantees the existence of a stepsize $\bar{\alpha} = O(\frac{1}{n^4})$ at each iteration of inexact IP-PMM, the proof follows \cite[Thm.~2]{PougkakiotisGondzio_2021_IPPMM}.
\end{proof}

\begin{remark}
Given \Cref{lem:bdd-step} and \Cref{lem:IPPMM-stepsize} in \Cref{sec:lemmas}, one can extend \cite[Thm.~1]{PougkakiotisGondzio_2021_IPPMM} and \cite[Thm.~3]{PougkakiotisGondzio_2021_IPPMM} to inexact IP-PMM using the parallel argument.
The former states that the duality measure $\mu_k$ converges $Q$-linearly to zero, and the latter guarantees global convergence to an optimal solution of (P)--(D).
\end{remark}

\subsection{Errors from the normal equations}
When $Q$ is diagonal, the normal equations \eqref{eq:normal-eq} are solved 
(inexactly) to determine the search direction,
and the other two components, $\Delta x^{k}$ and $\Delta z^{k}$, are obtained 
(exactly) using the closed-form formula \eqref{eq:delta_x}--\eqref{eq:delta_z}. 
Denote by $\inerrornormeq^{k}$ the residual in the normal equations at the $k$-th iteration:
\begin{equation} \label{eq:normal-eq-inexact}
[A (Q + \Theta_k^{-1} + \mu_k I_n)^{-1} A^T + \mu_k I_m] \Delta y^{k} = \xi^{k} + \inerrornormeq^{k}.
\end{equation}
The next result characterizes how the error $\inerrornormeq^{k}$ in \eqref{eq:normal-eq-inexact} propagates to the Newton system, which will be used in the later analysis for Nys-IP-PMM. The proof is presented in \Cref{sec:pf-prop}.

\begin{proposition} \label{prop:errors-propagation}
Suppose $(\Delta x^{k}, \Delta y^{k}, \Delta z^{k})$ satisfies \cref{eq:normal-eq-inexact,eq:delta_x,eq:delta_z} with $\| \inerrornormeq^{k} \|_2 \leq C \mu_k$ for some constant $C > 0$ and $\mu_k > 0$.
Then $(\Delta x^{k}, \Delta y^{k}, \Delta z^{k})$ satisfies \eqref{eq:inexact-Newton-RHS} with $(\inerror_{d}^{k}, \inerror_{p}^{k}, \inerror_{\mu}^{k}) = (0, \inerrornormeq^{k}, 0)$ and the residuals satisfy
\begin{equation} \label{eq:prop-error-goal}
\inerror_{\mu}^{k}=0, \quad\left\|\left(\inerror_{p}^{k}, \inerror_{d}^{k}\right)\right\|_2 \leq C \mu_k, \quad\left\|\left(\inerror_{p}^{k}, \inerror_{d}^{k}\right)\right\|_{\mathcal{A}} \leq C \mu_k.
\end{equation}
In particular, \Cref{assump:Newton-error-bound} is fulfilled if 
\begin{equation} \label{eq:delta-xi-bound} 
C \leq \frac{\sigma_{\min}}{4 \mu_0} \min\left\{ C_N, \gamma_{\mathcal{A}} \rho \right\}.
\end{equation}
\end{proposition}
\Cref{prop:errors-propagation} says that, if the search direction is obtained through solving the normal equations inexactly and using \eqref{eq:delta_x}--\eqref{eq:delta_z}, then only the second block of Newton system in \eqref{eq:inexact-Newton-RHS} is inexact. 
The first and third blocks of equations in Newton system remain exact due to the exact formulas in \eqref{eq:delta_x}--\eqref{eq:delta_z} for $\Delta x^{k}$ and $\Delta z^{k}$.

\section{Nystr\"om preconditioned IP-PMM (Nys-IP-PMM)} \label{sec:Nys-IP-PMM}

We introduce our main algorithm: Nystr\"om preconditioned IP-PMM (Nys-IP-PMM) in this section. Given the assumption that $Q$ in (P)--(D) is diagonal, Nys-IP-PMM uses the Nystr\"om PCG to solve the normal equations \eqref{eq:normal-eq}. 
\Cref{sec:conv-analysis-Nys-IP-PMM} provides a convergence analysis for Nys-IP-PMM, which specifies the preconditioner sketch size $\ell$ 
required for our convergence theory to hold.
We have released a reference implementation of Nys-IP-PMM in Julia \cite{bezanson2017julia}, discussed in \cref{sec:implementation}.

\subsection{Convergence analysis of Nys-IP-PMM} \label{sec:conv-analysis-Nys-IP-PMM}

The analysis of inexact IP-PMM in \cref{sec:IPPMM-convergence} shows that the convergence of inexact IP-PMM requires the inexact errors of Newton system \eqref{eq:inexact-Newton-RHS} to satisfy the bounds in \Cref{assump:Newton-error-bound}, which hold true if, by \Cref{prop:errors-propagation}, NysPCG finds a solution $\Delta y^{k}$ to the inexact normal equations \eqref{eq:normal-eq-inexact} satisfying \eqref{eq:delta-xi-bound}.
Indeed, when the sketch size $\ell_k$ for Nystr\"om PCG is chosen appropriately, the Nystr\"om PCG can achieve a solution such that $\| \inerrornormeq^{k} \|_2 \leq C \mu_k$, where $C$ satisfies \eqref{eq:delta-xi-bound}, in $O(\log \frac{n}{\mu_k})$ PCG iterations with high probability (see \Cref{lem:NysPCG-for-IPPMM} for the rigorous statement).
Hence we establish the following convergence results for Nys-IP-PMM. The proof appears in \Cref{sec:pf-Nys-IP-PMM-convergence}.

\begin{theorem} \label{thm:main}
Let $\epsilon \in (0,1)$ be a given error tolerance. 
Instate the assumptions of \Cref{thm:inexact-IPPMM-QP} and suppose the sketch size in Nys-IP-PMM is taken to be
\begin{equation} \label{eq:main-thm-sketch-size}
\ell_k \geq 8 \left( \sqrt{d_{\text{eff}}(N_k, \mu_k)} + \sqrt{32 \log(16 (k+2))} \right)^2.
\end{equation} 
Then,
\begin{enumerate}[(i)]
    \item with probability at least 0.9, Nystr\"om PCG runs at most $O\left(\log \frac{n}{\epsilon} \right)$ iterations in each Nys-IP-PMM iteration; \label{item:thm-main-i} 
    \item after $k = O(n^4 \log(\frac{1}{\epsilon}))$ Nys-IP-PMM iterations, the duality measure satisfies $\mu_k \leq \epsilon$.
    \label{item:thm-main-ii} 
\end{enumerate}
Hence, the total number of matvecs with $A$ required is at most $O\left(n^4 \log(\frac{1}{\epsilon}) \log(\frac{n}{\epsilon})\right)$.
\end{theorem}
The sketch size $\ell_{k}$ in \Cref{thm:main} grows with $k$ to ensure the number of inner iterations of PCG is sublinear in $n$.
In any given iteration, if $\ell_{k}$ is chosen smaller than the bound \eqref{eq:main-thm-sketch-size},
item 1) above may fail, but item 2) still holds: in particular, Nys-IP-PMM is still guaranteed to converge.
Our practical observations in \cref{sec:numerical-exp} indicate that a fixed sketch size typically performs well.

\subsection{Implementation of Nys-IP-PMM}
\label{sec:implementation}

Many implementation details de\-vi\-ate slightly from the theory to improve performance, following \cite{PougkakiotisGondzio_2021_IPPMM}; see supplement for details.
Pseudocode for Nys-IP-PMM appears in \Cref{alg:NysIPPMM-practical}.

\begin{algorithm}
\caption{Randomized Nystr\"om Preconditioned IP-PMM (Nys-IP-PMM)}
\label{alg:NysIPPMM-practical}
\begin{algorithmic}[1] 
\Require{QP data $A$, $b$, $c$, $u$ as in $(\tilde{\text{P}})$, feasibility and optimality tolerance tol}
\Ensure{approximate primal-dual solution $(\hat{x}, \hat{y}, \hat{z}, \hat{w}, \hat{s})$ to problem $(\tilde{\text{P}})$--$(\tilde{\text{D}})$}

\vspace{0.5pc}

\State Compute the initial point $(x^{(0)}, y^{(0)}, z^{(0)}, w^{(0)}, s^{(0)})$ as in \eqref{eq:initial-pt-pratical}.

\State Set initial PMM subproblem parameters: $\lambda^{(0)} = y^{(0)}$, $\zeta^{(0)} = x^{(0)}$, $\rho_0 = 8$, $\delta_0 = 8$.

\For{$k = 0, 1, 2, \ldots$}
    \State Check the termination criteria using \cite[Alg.~TC]{PougkakiotisGondzio_2021_IPPMM}.

    \vspace{0.25pc}
    \State Construct $N_k = A (Q + \Theta_k^{-1} + \rho_k I)^{-1} A^T$ as a linear operator.
    
    \State Choose sketch size $\ell_{k}$. 
    
    \State Construct inverse Nystr\"om preconditioner $P_k^{-1}$ by \Cref{alg:Nys-precond}.

    \State Update the iterate and duality measure by \Cref{alg:pred-corr-scheme} with $P_k^{-1}$.
    
    \State Update PMM parameters $\rho_k$, $\delta_k$, $\lambda^{k}$, $\zeta^{k}$ by \cite[Alg.~PEU]{PougkakiotisGondzio_2021_IPPMM}.
\EndFor
\end{algorithmic}
\end{algorithm}

\subsubsection{Free and box-constrained variables}
Nys-IP-PMM accepts QPs with free variables and box-constrained variables:
\begin{align*}
&{\small
\begin{array}{rll}
   (\tilde{\text{P}}): &\minf  & \displaystyle\frac{1}{2} x^T Q x + c^T x \\
   &\st  & Ax = b \\
   & & x_{\mathcal{F}} \text{: free},~ x_{\mathcal{I}} \geq 0, \\
   & & 0 \leq x_{\mathcal{J}} \leq u_{\J},
\end{array}} ~
{\small
\begin{array}{rll}
   (\tilde{\text{D}}): &\maxf  & \displaystyle -\frac{1}{2} x^T Q x + b^T y - u^T s \\
   & \st  & -Qx + A^T y + z - s = c, \\
   & & y : \text{free}, ~z_{\I} \geq 0, ~z_{\J} \geq 0, \\
   & & s_{\J} \geq 0,
\end{array}}
\end{align*}
where $\mathcal{F}$, $\mathcal{I}$, and $\mathcal{J}$ are, respectively, the sets of indices for free variables, non-negative variables, and box-constrained variables; and vector $u_{\J} \in \R^{|\J|}$ denotes the upper bounds.
For the ease of presentation, we assume the vectors $u, z, s$ all have length $n$ and satisfy $u_{\I \cup \F} = 0$, $z_{\F} = 0$, and $s_{\I \cup \F} = 0$; and we use the notation $\I\J = \I \cup \J$ and $\I\F = \I \cup \F$.

It is very important to directly model box-constrained variables,
as encoding them as general inequality constraints increases 
the size of the normal equations from $m$ to $m + |\J|$.
Moreover, many practical problems do have bounds on problem variables.

The PMM subproblem to be solved at each Nys-IP-PMM iteration becomes:
\begin{equation}
{\small
\label{eq:PMM-subproblem-free-box}
\begin{array}{ll}
\text{minimize}  & \displaystyle
\frac{1}{2} x^T Q x + c^T x + (\lambda^{k})^T (b - Ax) + \frac{1}{2 \delta_k} \|Ax - b\|_2^2 + \frac{\rho_k}{2} \|x - \zeta^{k}\|^2, \\ 
\text{subject to} & x_{\F} \text{: free},~ x_{\I} \geq 0, ~0 \leq x_{\J} \leq u_{\J}.
\end{array}
}
\end{equation}
We introduce a logarithmic barrier in the Lagrangian to enforce non-negativity and box con\-straints: $x_{\mathcal{I}} \geq 0$ and $0 \leq x_{\mathcal{J}} \leq u_{\J}$, and derive the corresponding linear system to determine the search direction (analogous to \eqref{eq:normal-eq}--\eqref{eq:delta_z}) 
in the following five variables: primal variable $x$, Lagrange multiplier $y$ (associated with $Ax = b$), dual variable $z_{\I\J}$ (associated with $x_{\I\J} \geq 0$), slack primal variable $w_{\J} \coloneqq u_{\J} - x_{\J}$, and dual variable $s_{\J}$ (associated with $x_{\J} \leq u_{\J}$).
Pseudocode for the search direction appears in \Cref{alg:solve-Newton}, which takes RHS vectors (explicitly defined later in \cref{sec:predictor-corrector}), current iterate, current parameters, and preconditioner as input.
A detailed derivation for \eqref{eq:normal-eq-free-box}--\eqref{eq:delta-s-free-box} is available in \cref{sec:derivation-free-box-newton} of the supplementary materials.

\begin{algorithm}
\caption{Solve Newton system}
\label{alg:solve-Newton}
\begin{algorithmic}[1]
\Require{RHS vectors $r_p^k$, $r_d^k$, $r_u^k$, $r_{xz}^k$, $r_{ws}^k$; current iterates $x^{k}, z^{k}, w^{k}, s^{k}$; parameters $\rho_k, \delta_k$; inverse preconditioner $P_{k}^{-1}$}
\Ensure{directions $\Delta x^{k}$, $\Delta y^{k}$, $\Delta z^{k}$, $\Delta w^{k}$, $\Delta s^{k}$ that satisfy \eqref{eq:Newton-free-box}}

\vspace{0.5pc}

\State Compute the auxiliary vectors $\xi_{\text{au}}^k$, $\xi^{k} \in \R^n$:
\begin{align*}
&\left(\xi_{\text{au}}^k\right)_j = 
\begin{cases}
0, &\text{if } j \in \F; \\
-(x^{k}_j)^{-1} (r_{xz}^k)_{j}, &\text{if } j \in \I; \\
-(x^{k}_j)^{-1} (r_{xz}^k)_{j} + (w^{k}_j)^{-1} (r_{ws}^k)_{j} -(w^{k}_j)^{-1} s^{k}_j (r_{u})_{j}, &\text{if } j \in \J,
\end{cases} \\
&\xi^{k} = r_p^k + A (Q + \Theta_k^{-1} + \rho_k I)^{-1} (r_d^k + \xi_{\text{au}}^k).
\end{align*}

\State Use PCG with preconditioner $P_k$ to solve $\Delta y^{k}$ from the normal equations:
\begin{equation} \label{eq:normal-eq-free-box}
[A (Q + \Theta_k^{-1} + \rho_k I_n)^{-1} A^T + \delta_k I_m] \Delta y^{k} = \xi^{k}.
\end{equation}

\State Compute $\Delta x^{k}$, $\Delta z^{k}$, $\Delta w^{k}$, $\Delta s^{k}$ from closed-form formulae:
\begin{align}
&\Delta x^{k} = (Q + \Theta_k^{-1} + \rho_k I)^{-1} (A^T \Delta y^{k} - r_d^k - \xi_{\text{au}}^k); \label{eq:delta-x-free-box} \\
&\Delta w^{k}_{\J} = (r_u^k)_{\J} - \Delta x^{k}_{\J}, \quad \Delta w^{k}_{\I\F} = 0; \label{eq:delta-w-free-box} \\
&\Delta z^{k}_{\I\J} = \diag(x^{k}_{\I\J})^{-1} (r_{xz}^k)_{\I\J} - \diag(z^{k}_{\I\J}) \Delta x^{k}_{\I\J}, \quad \Delta z^{k}_{\F} = 0; \label{eq:delta-z-free-box} \\
&\Delta s^{k}_{\J} = \diag(w^{k}_{\J})^{-1} (r_{ws}^k)_{\J} - \diag(s^{k}_{\J}) \Delta w^{k}_{\J}, \quad \Delta s^{k}_{\I\F} = 0. \label{eq:delta-s-free-box}
\end{align}
\end{algorithmic}
\end{algorithm}

\subsubsection{Initial point}
Our choice of initial point is inspired by 
Mehrotra's initial point \cite{Mehrotra_1992_implementation} and that in \cite{PougkakiotisGondzio_2021_IPPMM}.
A candidate point is first constructed by ignoring the non-negative constraints and then is modified to ensure the positivity of $x_{\I\J}$, $z_{\I\J}$, $w_{\J}$, $s_{\J}$.
See details in \cref{sec:initial-pt} of the supplementary materials.

\subsubsection{Mehrotra's predictor-corrector method}
\label{sec:predictor-corrector}

\begin{algorithm}[!h]
    \caption{Mehrotra's predictor-corrector method}
    \label{alg:pred-corr-scheme}
    \begin{algorithmic}[1]
    \Require{current iterate $x^{k}$, $y^{k}$, $z^{k}$, $w^{k}$, $s^{k}$; PMM parameters $\lambda^k$, $\zeta^k$, $\rho_k$, $\delta_k$; inverse preconditioner $P_{k}^{-1}$}
    \Ensure{new iterate $x^{k+1}$, $y^{k+1}$, $z^{k+1}$, $w^{k+1}$, $s^{k+1}$ and new duality measure $\mu_{k+1}$}
    
    \vspace{0.5pc}
    
    \State Compute predictor step $(\Delta_p x^{k}, \Delta_p y^{k}, \Delta_p z^{k}, \Delta_p w^{k}, \Delta_p s^{k})$ by \Cref{alg:solve-Newton} with:
    \begin{equation} \label{eq:predictor-RHS}
    \begin{aligned}
    &r_d^k = c + Qx^{k} - A^T y^{k} - z^{k} + s^{k} + \rho_k (x^{k} - \zeta^{k}), \\
    &r_p^k = b - A x^{k} - \delta_k (y^{k} - \lambda^{k}), \\
    &(r_u^k)_{\J} = u_{\J} - x^{k}_{\J} - w^{k}_{\J}, \quad (r_u^k)_{\I\F} = 0, \\
    &(r_{xz}^k)_{\I\J} = - \diag(x^{k}_{\I\J}) \diag(z^{k}_{\I\J}) \mathbbm{1}_{|\I\J|}, \quad (r_{xz}^k)_{\F} = 0, \\
    &(r_{ws}^k)_{\J} = - \diag(w^{k}_{\J}) \diag(s^{k}_{\J}) \mathbbm{1}_{|\J|}, \quad (r_{ws}^k)_{\I\F} = 0,.
    \end{aligned} 
    \end{equation}
    
    \State Find stepsizes $\bar{\alpha}_p$ and $\bar{\alpha}_d$ for predictor step using \eqref{eq:stepsize-p}--\eqref{eq:stepsize-d}.
    
    \State Compute the hypothetical measure for predictor step:
    \begin{equation} \label{eq:mu-aff}
    \mu_{\text{aff}} = \tfrac{(x^{k}_{\I\J} + \bar{\alpha}_p \Delta_p x^{k}_{\I\J} )^T (z^{k}_{\I\J} + \bar{\alpha}_d \Delta_p z^{k}_{\I\J}) + (w^{k}_{\J} + \bar{\alpha}_p \Delta_p w^{k}_{\J} )^T (s^{k}_{\J} + \bar{\alpha}_d \Delta_p s^{k}_{\J})}{|\I\J|+|\J|}.
    \end{equation}

    \State Compute the duality measure $\mu_k$ and $\tilde{\mu}^{k}$:
    \begin{equation*}
    \mu_k = \tfrac{\left(x^{k}_{\I\J}\right)^T z^{k}_{\I\J} + \left(w^{k}_{\J}\right)^T s^{k}_{\J}}{|\I\J|+|\J|}, \quad \tilde{\mu}^{k} \coloneqq \tfrac{\mu_{\text{aff}}^3}{\mu_k^2}.
    \end{equation*}
    
    \State Compute corrector step $(\Delta_c x^{k}, \Delta_c y^{k}, \Delta_c z^{k}, \Delta_c w^{k}, \Delta_c s^{k})$ by \Cref{alg:solve-Newton} with:
    \begin{equation} \label{eq:corrector-RHS}
    \begin{alignedat}{2}
    &r_d^k = 0, \quad 
    r_p^k = 0, \quad
    (r_u^k)_{\J} = 0, \\
    &(r_{xz}^k)_{\I\J} = \tilde{\mu}^{k} \mathbbm{1}_{|\I\J|} - \diag(\Delta_p x^{k}_{\I\J}) \diag(\Delta_p z^{k}_{\I\J}) \mathbbm{1}_{|\I\J|}, \quad &&(r_{xz}^k)_{\F} = 0, \\
    &(r_{ws}^k)_{\J} = \tilde{\mu}^{k} \mathbbm{1}_{|\J|} - \diag(\Delta_p w^{k}_{\J}) \diag(\Delta_p s^{k}_{\J}) \mathbbm{1}_{|\J|}, \quad &&(r_{ws}^k)_{\I\F} = 0.
    \end{alignedat}  
    \end{equation}

    \State Compute the final search direction:
    \begin{align*}
    &(\Delta x^{k}, \Delta w^{k}) = (\Delta_p x^{k} + \Delta_c x^{k}, \Delta_p w^{k} + \Delta_c w^{k}) \\
    &(\Delta y^{k}, \Delta z^{k}, \Delta s^{k}) = (\Delta_p y^{k} + \Delta_c y^{k}, \Delta_p z^{k} + \Delta_c z^{k}, \Delta_p s^{k} + \Delta_c s^{k})
    \end{align*}

    \State Find stepsizes $\alpha_{p}$ and $\alpha_{d}$ for final search direction using \eqref{eq:stepsize-p}--\eqref{eq:stepsize-d}.
    
    \State Update the iterates:
    \begin{align*}
    &(x^{k+1}, w^{k+1}) = (x^{k}, w^{k}) + \alpha_{p} (\Delta x^{k}, \Delta w^{k}) \\
    &(y^{k+1}, z^{k+1}, s^{k+1}) = (y^{k}, z^{k}, s^{k}) + \alpha_{d} (\Delta y^{k}, \Delta z^{k}, \Delta s^{k})
    \end{align*}
    
    \State Update the duality measure: $\mu_{k+1} = \frac{\left(x^{k+1}_{\I\J}\right)^T z^{k+1}_{\I\J} + \left(w^{k+1}_{\J}\right)^T s^{k+1}_{\J}}{|\I\J|+|\J|}$
    
    \end{algorithmic}
\end{algorithm}

Our Nys-IP-PMM method determines the search direction by Mehrotra's pre\-dic\-tor-corrector method \cite{Mehrotra_1992_implementation}, which is considered the industry-standard approach in IPMs.
The method automatically selects the centering parameter $\sigma_k$ by solving two Newton systems:
the first chooses $\sigma_k$, and the second chooses the search direction. See \Cref{alg:pred-corr-scheme}.
Fortunately, the two systems differ only in their RHS vectors,
so the preconditioner can be reused.

The first direction, called \textit{predictor step} and also known as \textit{affine-scaling direction}, aims at reducing the duality measure $\mu$ as much as possible and ignores centrality by setting $\sigma_k = 0$. Hence, the RHS vectors $r_p$, $r_d$, and $r_u$ are taken to be the negative residuals of primal, dual, and upper-bounded constraints respectively; while $r_{xz}$ and $r_{ws}$ are negative residuals of complementary slackness (see \eqref{eq:predictor-RHS}). 

Given the predictor step and associated stepsizes, Mehrotra suggests the centering parameter $\sigma_k = (\mu_{\text{aff}} / \mu_k)^3$, where $\mu_{\text{aff}}$ is the hypothetical duality measure resulting from making the predictor step (see \eqref{eq:mu-aff}).
The second direction, called the \textit{corrector step}, consists of centering component and correcting component that attempts to compensate for error in the linear approximation to complementary slackness.
See \eqref{eq:corrector-RHS} for the associated RHS vectors.
The final search direction is the sum of predictor step and corrector step.

\section{Numerical experiments} \label{sec:numerical-exp}
We demonstrate the effectiveness of Nys-IP-PMM numerically in this section.
We compare Nys-IP-PMM with other matrix-free IP-PMMs: 
IP-PMM using CG (CG-IP-PMM)
and IP-PMM using PCG with the partial Cholesky preconditioner 
(Chol-IP-PMM).
In \cref{sec:portfolio}, we show that matrix-free algorithms can handle a large-scale portfolio optimization problem that is too large for IPMs using a direct inner linear system solver. 
Among the matrix-free methods, Nys-IP-PMM is fastest.
\Cref{sec:SVM} compares these matrix-free methods on linear support vector machine (SVM) problems \cite{fine2001efficient, woodsend2011exploiting, gondzio2009exploiting,  woodsend2009hybrid}. 
Nys-IP-PMM beats Chol-IP-PMM substantially. It also improves on CG-IP-PMM when the constraint matrix $A$ is dense, and is never much slower.

All experiments were run on a server with 2x 64-core AMD EPYC 7763 @ 2.45GHz, 1 TB RAM. We only use 16 cores for any given experiment since the calls to BLAS are limited to 16 threads. 
The implementation uses Julia v1.9.1 \cite{bezanson2017julia}. Code is publicly available at \url{https://github.com/udellgroup/Nys-IP-PMM}. Input matrices $Q$ and $A$ are passed as linear operator objects using \texttt{LinearOperators.jl} \cite{dominique_orban_2024_10713080}, while \texttt{Krylov.jl} \cite{montoison-orban-2023} is selected, because of its  efficiency, as the implementation for the con\-ju\-gate gradient (CG) method.
When entrywise access to $A$ is available, our implementation for the partial Cholesky preconditioner 
provides an option to compute the diagonal of $N_k$ more efficiently (see footnote \ref{fn:entrywise}). 
Our experiments use this option for Chol-IP-PMM.
Hence the performance of Chol-IP-PMM is \emph{better} than it would be in the matrix-free setting,
understating the advantage of Nys-IP-PMM.

\subsection{Large-scale portfolio optimization problem} \label{sec:portfolio}

This experiment show\-cases the use of Nys-IP-PMM for large-scale separable QPs. We construct a synthetic portfolio optimization problem, which aims at determining the asset allocation to maximize risk-adjusted returns while constraining correlation with market indexes or competing portfolios:
\begin{equation} \label{eq:portfolio}
\begin{array}{cl}
   \minf  & \displaystyle - r^T x + \gamma x^T \Sigma x \\
   \st  & M x \leq u, ~ \mathbbm{1}_n^T x = 1, ~ x \geq 0,
\end{array}
\end{equation}
where variable $x \in \mathbb{R}^n$ represents the portfolio, $r \in \mathbb{R}^n$ denotes the vector of expected returns, $\gamma > 0$ denotes the risk aversion parameter, $\Sigma \in \mathbb{S}_n^{+}(\mathbb{R})$ represents the risk model covariance matrix, each row of $M \in \mathbb{R}^{d \times n}$ represents another portfolio, and $u \in \mathbb{R}^{d}$ upper bounds of the correlations.
We assume a factor model for the covariance matrix $\Sigma = F F^T + D$, where $F \in \R^{n \times s}$ is the factor loading matrix and $D \in \R^{n \times n}$ is a diagonal matrix representing asset-specific risk, and solve the equivalent reformulation of \eqref{eq:portfolio}. See \Cref{sec:factor-model} for details.

In our experiment, a true covariance matrix $\Sigma^\natural$ is generated with rapid spectral decay followed by slow decay; the rows of $M$ are normally distributed; the expected returns $r$ are sampled from standard normal; correlation upper bounds $u$ are sampled from a uniform distribution, and $\gamma$ is set to $1$.
To obtain the risk model $\Sigma$, we use the randomized rank-$s$ Nystr\"om approximation on $\Sigma^\natural$ to compute the factors $F$ and define $D$ by $\diag(D) = \diag(\Sigma - F F^T)$. 
The demonstrated problem instance has dimensions $n = \num{80000}$, $d = \num{50000}$, and $s = \num{100}$. 
The normal equations to solve at each iteration have size $d+s+1=\num{50101}$.
The tolerance for relative primal-dual infeasibility and optimality gap $\mu$ is set as $\epsilon = 10^{-8}$. 
\Cref{fig:portfolio} shows history of these three convergence measurements versus wallclock time.  
Nys-IP-PMM with sketchsize $\ell = 20$ converges $2$x faster than CG-IP-PMM and more than $3$x faster than Chol-IP-PMM.

\begin{figure}
    \centering
    \includegraphics[width=\linewidth]{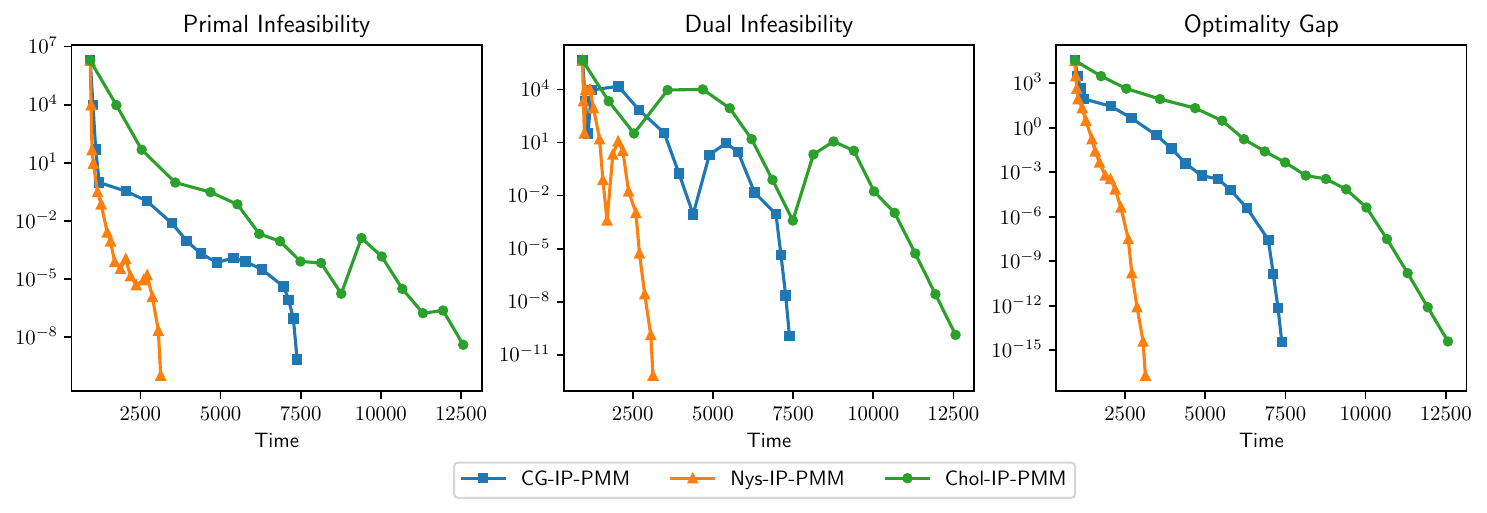}
    \caption{Relative primal/dual infeasibility and optimality gap versus cumulative time for portfolio optimization problem with $n = \num{80000}$, $d = \num{50000}$, and $s = 100$. 
    }
    \label{fig:portfolio}
\end{figure}

\subsection{Support vector machine (SVM) problem} \label{sec:SVM}
The linear support vector machine (SVM) problem solves a binary classification task on $n$ samples with $d$ features \cite[Chap.~12]{deisenroth2020mathematics}. 
It is well-known that dual linear SVM with $\ell_1$-regularization can be formulated as a convex QP \cite{fine2001efficient, woodsend2011exploiting, gondzio2009exploiting, woodsend2009hybrid}. See the formulation we used in \cref{sec:SVM-formulation}.
In this example, the constraint matrix $A$ contains the feature matrix as a block and thus can be very dense.  
We use the medium-to-large sized real datasets from UCI \cite{frank2010uci} and LIBSVM \cite{chang2011libsvm} listed in Table \ref{tab:datasets}.
\begin{table}[h]
\centering
\caption{SVM Datasets information.}
\label{tab:datasets}
\resizebox{\linewidth}{!}{
\begingroup
\begin{tabular}{@{}cccccccc@{}}
\toprule
Datasets & Features $d$ & Instances $n$ & nnz \% & Datasets & Features $d$ & Instances $n$ & nnz \% \\ \cmidrule(r){1-4} \cmidrule(l){5-8}
SensIT & \num{100} & \num{98528} & 100.0 & Arcene & \num{10000} & 100 & 54.1 \\
CIFAR10 & \num{3072} & \num{60000} & 99.7 & Dexter & \num{20000} & 300 & 0.5 \\
STL-10 & \num{27648} & \num{13000} & 96.3 & Sector & \num{55197} & \num{9619} & 0.3 \\
RNASeq & \num{20531} & \num{801} & 85.8 \\
\bottomrule
\end{tabular}
\endgroup
}
\end{table}

\subsubsection{SVM results}

\begin{table}[h]
\centering
\caption{Nys-IP-PMM vs other preconditioners on SVM problem.}
\label{tab:SVM-results}
\resizebox{\linewidth}{!}{
\begingroup
\begin{tabular}{@{}ccccccccccc@{}}
\toprule
\multicolumn{1}{l}{\multirow{3}{*}{Datasets}} & \multicolumn{3}{c}{\textbf{Nys-IP-PMM}} & \multicolumn{3}{c}{\textbf{CG-IP-PMM}} & \multicolumn{3}{c}{\textbf{Chol-IP-PMM}} & \\ 

 & \begin{tabular}[c]{@{}c@{}}IP-PMM\\ Outer Iter.\end{tabular} & \begin{tabular}[c]{@{}c@{}}Sum of\\ Inner Iter.\end{tabular} & \begin{tabular}[c]{@{}c@{}}Time\\ (sec.)\end{tabular}  & \begin{tabular}[c]{@{}c@{}}IP-PMM\\ Outer Iter.\end{tabular} & \begin{tabular}[c]{@{}c@{}}Sum of\\ Inner Iter.\end{tabular} & \begin{tabular}[c]{@{}c@{}}Time\\ (sec.)
 \end{tabular} & \begin{tabular}[c]{@{}c@{}}IP-PMM\\ Outer Iter.\end{tabular} & \begin{tabular}[c]{@{}c@{}}Sum of\\ Inner Iter.\end{tabular} & \begin{tabular}[c]{@{}c@{}}Time\\ (sec.)
 \end{tabular} & \begin{tabular}[c]{@{}c@{}}Rank\\ $\ell_k \equiv \ell$
 \end{tabular} \\ 
 \cmidrule(lr){2-4} \cmidrule(lr){5-7} \cmidrule(lr){8-10} \cmidrule(ll){11-11}

CIFAR10 &  12   & \num{18692}  &  $\bf{1103.26}$      & $\ast$   & $\ast$ & $\ast$ & $\ast$   & $\ast$ & $\ast$ & 200 \\
RNASeq  &  14   &  \num{4793}  &   {\bf 8.40}     &  14  &  \num{24689}  &  19.59  & 14  &  \num{20488} &  \num{769.95} & 200 \\
STL-10  &  12   & \num{119760}  & $\bf{7691.41}$       &  12  & \num{492925}  &  \num{30016.8}  & 12  & \num{491503} &  \num{103681.65} & 800  \\
SensIT  &  15   &  \num{1621}  &    \textbf{5.69}    &  15  &   \num{4569}  &  14.12  & 15  & \num{1164} &  11.05 & 50 \\
Arcene  &   5   & \num{386} & $\bf{0.233}$  &   5  &    \num{649}  &  \num{0.278} &  5  & \num{6194} &  18.15 & 20 \\
Dexter  &   4   &  \num{344}  &  \num{0.202}   &   4  &    \num{333}  &  $\bf{0.122}$ &  4  & \num{3754} &  20.81 & 10 \\
Sector  &  16   &  \num{3570}  & \textbf{16.33}     &  16  &   \num{4493}  &  16.99  & 16  & \num{6940} &  \num{681.96} & 20 \\
\bottomrule
\end{tabular}
\endgroup
}

\end{table} 

This section evaluates the performance of Nys-IP-PMM on the SVM problem \eqref{eq:dual-SVM} using datasets listed in \Cref{tab:datasets}.
We solve SVM problems with Nys-IP-PMM, CG-IP-PMM, and Chol-IP-PMM.
All three methods are matrix-free, meaning they do not rely on direct access to the entries of $N_k + \delta_k I$.
The results, presented in \Cref{tab:SVM-results}, highlight the effectiveness Nys-IP-PMM. 
In contrast, Chol-IP-PMM does not improve runtime as much,
and in some cases, can even increase the required number of inner iterations (e.g., for Arcene, Dexter, and Sector). 
Its construction time can also be significant.
Overall, Chol-IP-PMM is often much slower than CG-IP-PMM.

The randomized Nystr\"om preconditioner accelerates IP-PMM more effectively on dense datasets such as SensIT, CIFAR10, STL-10, and RNASeq. As indicated in \Cref{tab:comparison}, the construction time of the Nystr\"om preconditioner differs significantly from that of the partial Cholesky preconditioner when $A$ is dense.
Datasets Arcene and Dexter are relatively easy, requiring few inner iterations even without a preconditioner.

In summary, the randomized Nystr\"om preconditioner is a competitive choice for SVM problems. It effectively preconditions ill-conditioned instances without much slowdown for well-conditioned instances.

\subsubsection{Condition numbers at different IP-PMM stages} \label{sec:exp-cond}

\begin{figure} \label{fig:cond_rank}
    \centering
    \includegraphics[width=\linewidth]{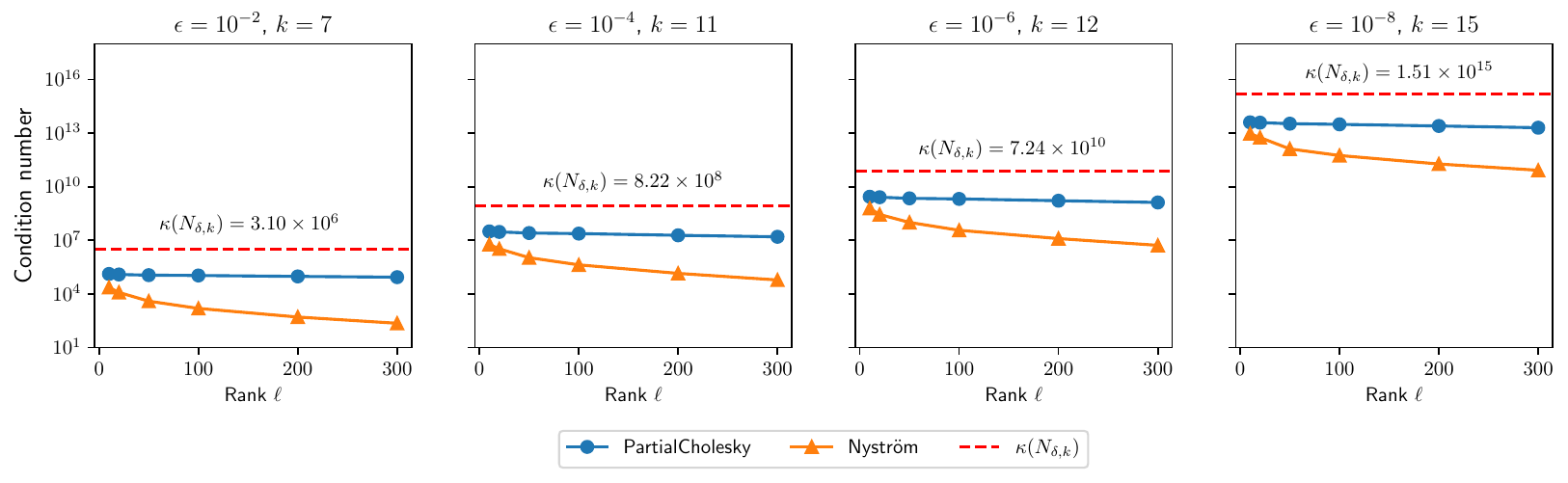}
    \caption{The condition numbers before and after preconditioning.
    The subplots represent distinct stages of IP-PMM convergence. The red dashed line shows the unpreconditioned condition number $\kappa(N_{\delta, k})$.
    Blue circles denote the condition number after partial Cholesky preconditioning, while orange triangles represent the condition number after Nystr\"om preconditioning.
    }
\end{figure}

We explore the impact of the preconditioner across various stages of IP-PMM by examining the condition number.
We record the regularized normal equations \eqref{eq:normal-eq} in the iteration of which IP-PMM attains primal-dual infeasibility and duality measure less than different tolerance levels $\epsilon \in \{10^{-2}, 10^{-4}, 10^{-6}, 10^{-8}\}$.
Then we compute the condition number of $N_{\delta} = N_k + \delta_k I$ before and after applying different preconditioners.
For efficiency, this experiment 
concerns a smaller SVM problem
formed from $1000$ samples from CIFAR10;
the normal equations have size $3073 \times 3073$.

\Cref{fig:cond_rank} demonstrates the (preconditioned) condition numbers with different choices $\ell \in \{10, 20, 50, 100, 200, 300\}$.
The red dashed horizontal line denotes the condition number $\kappa(N_{\delta})$ before preconditioning.
The blue circles denote the condition numbers after partial Cholesky preconditioning,
while the orange triangles denote those after Nystr\"om preconditioning.
As expected, the normal equation matrix $N_{\delta}$ becomes more and more ill-conditioned as IP-PMM converges. 
The preconditioners have similar effects at each stage of convergence.
Both preconditioners reduce the condition number by 1--2 orders of magnitude.
As the rank $\ell$ increases, the Nystr\"om preconditioner improves the condition number by another 1--2 orders of magnitude,
while larger ranks barely improve the effectiveness 
of partial Cholesky preconditioner.

\subsubsection{Rank $\ell$ versus total/construction/PCG Time}

We illustrate trade-offs between the rank $\ell$ and the wallclock time of Nys-IP-PMM and Chol-IP-PMM. 
\Cref{fig:rank_time} shows the averaged wallclock time over $4$ independent runs for SVM problem on the RNASeq dataset, where the normal equations have dimension $\num{20532} \times \num{20532}$.

The left subplot of \Cref{fig:rank_time} shows that the Nystr\"om preconditioner reduces overall runtime dramatically compared to the partial Cholesky preconditioner. 
The construction time for the partial Cholesky preconditioner increases with increased rank (as expected), 
while the PCG runtime does not decrease and even increases.
We see that partial Cholesky is not an effective pre\-con\-di\-tioner for this problem.

Zooming in, the right subplot of \Cref{fig:rank_time} shows the trade-off between the rank parameter $\ell$ and the total time for Nys-IP-PMM. 
The PCG runtime dominates for smaller ranks, while the time to construct the preconditioner dominates for higher ranks.
The overall time exhibits a U-shape: 
excessively small or large ranks lead to slower computational times.
Further improvements to Nys-IP-PMM might be achieved by tuning the rank for each problem and by reusing the preconditioner between iterations.

\begin{figure}
    \centering
    \includegraphics[width=\linewidth]{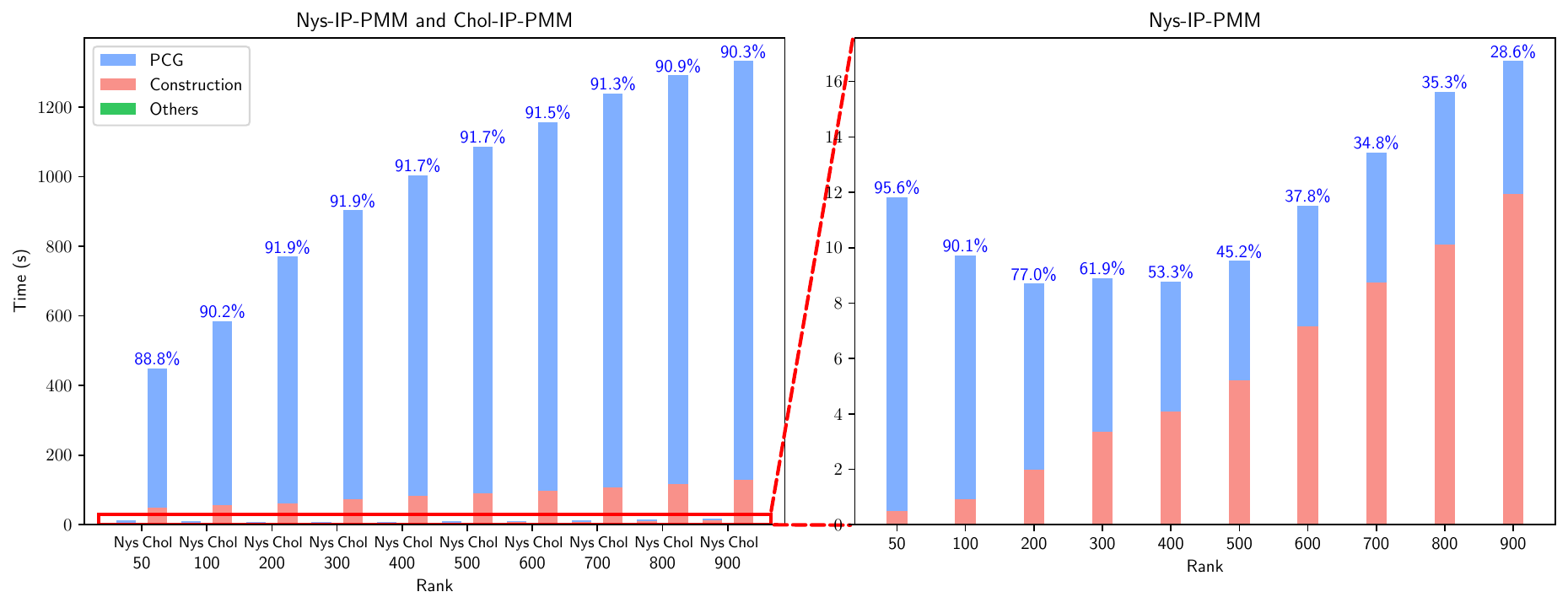}
    \caption{Runtime of Nys-IP-PMM and Chol-IP-PMM with varying rank $\ell$ 
    on RNASeq dataset. 
    Left plot compares the two methods; right plot zooms in on the bar chart for Nys-IP-PMM.     
    Bar height show total runtime, which is broken into 1) PCG runtime, 2) construction time for preconditioner, and 3) other computation (which is negligible). Blue number above bar gives percentage PCG time. 
    }
    \label{fig:rank_time}
\end{figure}

\section{Concluding remarks}
We present a new algorithm, Nys-IP-PMM, for large-scale separable QP by solving the normal equations in IP-PMM using randomized Nystr\"om preconditioned conjugate gradient (PCG). 
Nys-IP-PMM leverages matrix-free precondition\-ing for computational efficiency. 
Theoretical analysis and numerical experiments demonstrate superiority of Nys-IP-PMM over existing matrix-free IPM-based me\-thods. 
Our open-source implementation provides a flexible basis for further exploration of matrix-free regularized IPMs, 
including new preconditioners, extensions to hardware accelerators such as GPUs, and IPMs for non-separable QP 
(which require using the augmented system)
or other classes of optimization problems.
By combining theoretical insights with empirical validation, 
our work advances the state-of-the-art in large-scale optimization, with potential applications in finance, engineering, and machine learning.

\bibliographystyle{siamplain}
\bibliography{references}

\begin{thebibliography}{10}

\bibitem{altman1999regularized}
{\sc A.~Altman and J.~Gondzio}, {\em Regularized symmetric indefinite systems
  in interior point methods for linear and quadratic optimization},
  Optimization Methods and Software, 11 (1999), pp.~275--302,
  \url{https://doi.org/10.1080/10556789908805754}.

\bibitem{bellavia2013matrix}
{\sc S.~Bellavia, J.~Gondzio, and B.~Morini}, {\em A {{Matrix-Free
  Preconditioner}} for {{Sparse Symmetric Positive Definite Systems}} and
  {{Least-Squares Problems}}}, SIAM Journal on Scientific Computing, 35 (2013),
  pp.~A192--A211, \url{https://doi.org/10.1137/110840819}.

\bibitem{bergamaschi2021new}
{\sc L.~Bergamaschi, J.~Gondzio, {\'A}.~Mart{\'i}nez, J.~W. Pearson, and
  S.~Pougkakiotis}, {\em A new preconditioning approach for an interior
  point-proximal method of multipliers for linear and convex quadratic
  programming}, Numerical Linear Algebra with Applications, 28 (2021),
  p.~e2361, \url{https://doi.org/10.1002/nla.2361}.

\bibitem{bezanson2017julia}
{\sc J.~Bezanson, A.~Edelman, S.~Karpinski, and V.~B. Shah}, {\em Julia: {{A
  Fresh Approach}} to {{Numerical Computing}}}, SIAM Review, 59 (2017),
  pp.~65--98, \url{https://doi.org/10.1137/141000671}.

\bibitem{bocanegra2007using}
{\sc S.~Bocanegra, F.~F. Campos, and A.~R.~L. Oliveira}, {\em Using a hybrid
  preconditioner for solving large-scale linear systems arising from interior
  point methods}, Computational Optimization and Applications, 36 (2007),
  pp.~149--164, \url{https://doi.org/10.1007/s10589-006-9009-5}.

\bibitem{casacio2017improving}
{\sc L.~Casacio, C.~Lyra, A.~R.~L. Oliveira, and C.~O. Castro}, {\em Improving
  the preconditioning of linear systems from interior point methods}, Computers
  \& Operations Research, 85 (2017), pp.~129--138,
  \url{https://doi.org/10.1016/j.cor.2017.04.005}.

\bibitem{chan1988optimal}
{\sc T.~F. Chan}, {\em An {{Optimal Circulant Preconditioner}} for {{Toeplitz
  Systems}}}, SIAM Journal on Scientific and Statistical Computing, 9 (1988),
  pp.~766--771, \url{https://doi.org/10.1137/0909051}.

\bibitem{chang2011libsvm}
{\sc C.-C. Chang and C.-J. Lin}, {\em {{LIBSVM}}: {{A}} library for support
  vector machines}, ACM Transactions on Intelligent Systems and Technology, 2
  (2011), pp.~1--27, \url{https://doi.org/10.1145/1961189.1961199}.

\bibitem{Chowdhury_2022_RandIPM}
{\sc A.~Chowdhury, G.~Dexter, P.~London, H.~Avron, and P.~Drineas}, {\em Faster
  randomized interior point methods for {{Tall}}/{{Wide}} linear programs},
  Journal of Machine Learning Research, 23 (2022), pp.~1--48,
  \url{http://jmlr.org/papers/v23/21-0709.html}.

\bibitem{vciegis2009high}
{\sc R.~{\v{C}}iegis, D.~Henty, B.~K{\aa}gstr{\"o}m, J.~{\v{Z}}ilinskas,
  K.~Woodsend, and J.~Gondzio}, {\em High-performance parallel support vector
  machine training}, Parallel Scientific Computing and Optimization: Advances
  and Applications,  (2009), pp.~83--92.

\bibitem{de2017condition}
{\sc F.~{de Prenter}, C.~V. Verhoosel, G.~J. {van Zwieten}, and E.~H. {van
  Brummelen}}, {\em Condition number analysis and preconditioning of the finite
  cell method}, Computer Methods in Applied Mechanics and Engineering, 316
  (2017), pp.~297--327, \url{https://doi.org/10.1016/j.cma.2016.07.006}.

\bibitem{deisenroth2020mathematics}
{\sc M.~P. Deisenroth, A.~A. Faisal, and C.~S. Ong}, {\em Mathematics for
  machine learning}, Cambridge University Press, 2020.

\bibitem{dexter2022convergence}
{\sc G.~Dexter, A.~Chowdhury, H.~Avron, and P.~Drineas}, {\em On the
  convergence of inexact predictor-corrector methods for linear programming},
  in International Conference on Machine Learning, PMLR, 2022, pp.~5007--5038,
  \url{https://proceedings.mlr.press/v162/dexter22a.html}.

\bibitem{durazzi2003indefinitely}
{\sc C.~Durazzi and V.~Ruggiero}, {\em Indefinitely preconditioned conjugate
  gradient method for large sparse equality and inequality constrained
  quadratic problems}, Numerical Linear Algebra with Applications, 10 (2003),
  pp.~673--688, \url{https://doi.org/10.1002/nla.308}.

\bibitem{fine2001efficient}
{\sc S.~Fine and K.~Scheinberg}, {\em Efficient {SVM} training using low-rank
  kernel representations}, Journal of Machine Learning Research, 2 (2001),
  pp.~243--264, \url{https://www.jmlr.org/papers/v2/fine01a.html}.

\bibitem{forsgren2002interior}
{\sc A.~Forsgren, P.~E. Gill, and M.~H. Wright}, {\em Interior {{Methods}} for
  {{Nonlinear Optimization}}}, SIAM Review, 44 (2002), pp.~525--597,
  \url{https://doi.org/10.1137/S0036144502414942}.

\bibitem{Frangella_2023_RandNysPCG}
{\sc Z.~Frangella, J.~A. Tropp, and M.~Udell}, {\em Randomized {{Nystr{\"o}m
  Preconditioning}}}, SIAM Journal on Matrix Analysis and Applications, 44
  (2023), pp.~718--752, \url{https://doi.org/10.1137/21M1466244}.

\bibitem{friedlander2012primal}
{\sc M.~P. Friedlander and D.~Orban}, {\em A primal--dual regularized
  interior-point method for convex quadratic programs}, Mathematical
  Programming Computation, 4 (2012), pp.~71--107,
  \url{https://doi.org/10.1007/s12532-012-0035-2}.

\bibitem{gill1986projected}
{\sc P.~E. Gill, W.~Murray, M.~A. Saunders, J.~A. Tomlin, and M.~H. Wright},
  {\em On projected newton barrier methods for linear programming and an
  equivalence to {{Karmarkar}}'s projective method}, Mathematical Programming,
  36 (1986), pp.~183--209, \url{https://doi.org/10.1007/BF02592025}.

\bibitem{goldfarb1990n}
{\sc D.~Goldfarb and S.~Liu}, {\em An $\mathcal{O}(n^3 \mathrm{L})$ primal
  interior point algorithm for convex quadratic programming}, Mathematical
  Programming, 49 (1990), pp.~325--340,
  \url{https://doi.org/10.1007/BF01588795}.

\bibitem{golub2013matrix}
{\sc G.~H. Golub and C.~F. Van~Loan}, {\em Matrix Computations}, Johns
  {{Hopkins}} Studies in the Mathematical Sciences, The Johns Hopkins
  University Press, Baltimore, 4~ed., 2013.

\bibitem{gondzio2012_IPM25}
{\sc J.~Gondzio}, {\em Interior point methods 25 years later}, European Journal
  of Operational Research, 218 (2012), pp.~587--601,
  \url{https://doi.org/10.1016/j.ejor.2011.09.017}.

\bibitem{Gondzio_2012_mfIPM}
{\sc J.~Gondzio}, {\em Matrix-free interior point method}, Computational
  Optimization and Applications, 51 (2012), pp.~457--480,
  \url{https://doi.org/10.1007/s10589-010-9361-3}.

\bibitem{gondzio2009exploiting}
{\sc J.~Gondzio and A.~Grothey}, {\em Exploiting structure in parallel
  implementation of interior point methods for optimization}, Computational
  Management Science, 6 (2009), pp.~135--160,
  \url{https://doi.org/10.1007/s10287-008-0090-3}.

\bibitem{gondzio2022general}
{\sc J.~Gondzio, S.~Pougkakiotis, and J.~W. Pearson}, {\em General-purpose
  preconditioning for regularized interior point methods}, Computational
  Optimization and Applications, 83 (2022), pp.~727--757,
  \url{https://doi.org/10.1007/s10589-022-00424-5}.

\bibitem{hestenes1969multiplier}
{\sc M.~R. Hestenes}, {\em Multiplier and gradient methods}, Journal of
  Optimization Theory and Applications, 4 (1969), pp.~303--320,
  \url{https://doi.org/10.1007/BF00927673}.

\bibitem{johnson1983polynomial}
{\sc O.~G. Johnson, C.~A. Micchelli, and G.~Paul}, {\em Polynomial
  {{Preconditioners}} for {{Conjugate Gradient Calculations}}}, SIAM Journal on
  Numerical Analysis, 20 (1983), pp.~362--376,
  \url{https://doi.org/10.1137/0720025}.

\bibitem{kapoor1986fast}
{\sc S.~Kapoor and P.~M. Vaidya}, {\em Fast algorithms for convex quadratic
  programming and multicommodity flows}, in Proceedings of the Eighteenth
  Annual {{ACM}} Symposium on {{Theory}} of Computing - {{STOC}} '86, Berkeley,
  1986, ACM Press, pp.~147--159, \url{https://doi.org/10.1145/12130.12145}.

\bibitem{karmarkar1984new}
{\sc N.~Karmarkar}, {\em A new polynomial-time algorithm for linear
  programming}, Combinatorica, 4 (1984), pp.~373--395,
  \url{https://doi.org/10.1007/BF02579150}.

\bibitem{frank2010uci}
{\sc M.~Kelly, R.~Longjohn, and K.~Nottingham}, {\em {The UCI Machine Learning
  Repository}}, \url{http://archive.ics.uci.edu} (accessed 2024/04/14).

\bibitem{kojima1993primal}
{\sc M.~Kojima, N.~Megiddo, and S.~Mizuno}, {\em A primal-dual
  infeasible-interior-point algorithm for linear programming}, Mathematical
  Programming, 61 (1993), pp.~263--280,
  \url{https://doi.org/10.1007/BF01582151}.

\bibitem{kojima1989primal}
{\sc M.~Kojima, S.~Mizuno, and A.~Yoshise}, {\em A {{Primal-Dual Interior Point
  Algorithm}} for {{Linear Programming}}}, in Progress in {{Mathematical
  Programming}}: {{Interior-Point}} and {{Related Methods}}, N.~Megiddo, ed.,
  Springer, New York, NY, 1989, pp.~29--47,
  \url{https://doi.org/10.1007/978-1-4613-9617-8_2}.

\bibitem{lacotte2020effective}
{\sc J.~Lacotte and M.~Pilanci}, {\em Effective dimension adaptive sketching
  methods for faster regularized least-squares optimization}, in Proceedings of
  the 34th {{International Conference}} on {{Neural Information Processing
  Systems}}, {{NIPS}} '20, Red Hook, NY, USA, Dec. 2020, Curran Associates
  Inc., pp.~19377--19387.

\bibitem{lobo1998applications}
{\sc M.~S. Lobo, L.~Vandenberghe, S.~Boyd, and H.~Lebret}, {\em Applications of
  second-order cone programming}, Linear Algebra and its Applications, 284
  (1998), pp.~193--228, \url{https://doi.org/10.1016/S0024-3795(98)10032-0}.

\bibitem{lustig1992implementing}
{\sc I.~J. Lustig, R.~E. Marsten, and D.~F. Shanno}, {\em On {{Implementing
  Mehrotra}}'s {{Predictor}}--{{Corrector Interior-Point Method}} for {{Linear
  Programming}}}, SIAM Journal on Optimization, 2 (1992), pp.~435--449,
  \url{https://doi.org/10.1137/0802022}.

\bibitem{lustig1994interior}
{\sc I.~J. Lustig, R.~E. Marsten, and D.~F. Shanno}, {\em Interior {{Point
  Methods}} for {{Linear Programming}}: {{Computational State}} of the
  {{Art}}}, ORSA Journal on Computing, 6 (1994), pp.~1--14,
  \url{https://doi.org/10.1287/ijoc.6.1.1}.

\bibitem{marsten1990interior}
{\sc R.~Marsten, R.~Subramanian, M.~Saltzman, I.~Lustig, and D.~Shanno}, {\em
  Interior {{Point Methods}} for {{Linear Programming}}: {{Just Call Newton}},
  {{Lagrange}}, and {{Fiacco}} and {{McCormick}}!}, Interfaces, 20 (1990),
  pp.~105--116, \url{https://doi.org/10.1287/inte.20.4.105}.

\bibitem{martinsson2020randomized}
{\sc P.-G. Martinsson and J.~A. Tropp}, {\em Randomized numerical linear
  algebra: {{Foundations}} and algorithms}, Acta Numerica, 29 (2020),
  pp.~403--572, \url{https://doi.org/10.1017/S0962492920000021}.

\bibitem{megiddo1989pathways}
{\sc N.~Megiddo}, {\em Pathways to the {{Optimal Set}} in {{Linear
  Programming}}}, in Progress in Mathematical Programming, Springer, New York,
  NY, 1989, pp.~131--158, \url{https://doi.org/10.1007/978-1-4613-9617-8_8}.

\bibitem{Mehrotra_1992_implementation}
{\sc S.~Mehrotra}, {\em On the implementation of a primal-dual interior point
  method}, SIAM Journal on optimization, 2 (1992), pp.~575--601,
  \url{https://doi.org/10.1002/nla.2361}.

\bibitem{mehrotra1990algorithm}
{\sc S.~Mehrotra and J.~Sun}, {\em An {{Algorithm}} for {{Convex Quadratic
  Programming That Requires O}}(n3.{{5L}}) {{Arithmetic Operations}}},
  Mathematics of Operations Research, 15 (1990), pp.~342--363,
  \url{https://doi.org/10.1287/moor.15.2.342}.

\bibitem{mizuno1999global}
{\sc S.~Mizuno and F.~Jarre}, {\em Global and polynomial-time convergence of an
  infeasible-interior-point algorithm using inexact computation}, Mathematical
  Programming, 84 (1999), pp.~105--122,
  \url{https://doi.org/10.1007/s10107980020a}.

\bibitem{monteiro1989interior}
{\sc R.~D.~C. Monteiro and I.~Adler}, {\em Interior path following primal-dual
  algorithms. {{Part I}}: {{Linear}} programming}, Mathematical Programming, 44
  (1989), pp.~27--41, \url{https://doi.org/10.1007/BF01587075}.

\bibitem{montoison-orban-2023}
{\sc A.~Montoison and D.~Orban}, {\em {Krylov.jl: A Julia basket of hand-picked
  Krylov methods}}, Journal of Open Source Software, 8 (2023), p.~5187,
  \url{https://doi.org/10.21105/joss.05187}.

\bibitem{morini2018partial}
{\sc B.~Morini}, {\em On {{Partial Cholesky Factorization}} and a {{Variant}}
  of {{Quasi-Newton Preconditioners}} for {{Symmetric Positive Definite
  Matrices}}}, Axioms, 7 (2018), p.~44,
  \url{https://doi.org/10.3390/axioms7030044}.

\bibitem{nesterov1994interior}
{\sc Y.~Nesterov and A.~Nemirovskii}, {\em Interior-{{Point Polynomial
  Algorithms}} in {{Convex Programming}}}, Studies in {{Applied}} and
  {{Numerical Mathematics}}, {SIAM}, Jan. 1994,
  \url{https://doi.org/10.1137/1.9781611970791}.

\bibitem{oliveira2005new}
{\sc A.~R.~L. Oliveira and D.~C. Sorensen}, {\em A new class of preconditioners
  for large-scale linear systems from interior point methods for linear
  programming}, Linear Algebra and its Applications, 394 (2005), pp.~1--24,
  \url{https://doi.org/doi: 10.1016/j.laa.2004.08.019}.

\bibitem{dominique_orban_2024_10713080}
{\sc D.~Orban and A.~S. Siqueira}, {\em {LinearOperators.jl}}, Feb. 2024,
  \url{https://doi.org/10.5281/zenodo.10713080},
  \url{https://github.com/JuliaSmoothOptimizers/LinearOperators.jl}.

\bibitem{PougkakiotisGondzio_2021_IPPMM}
{\sc S.~Pougkakiotis and J.~Gondzio}, {\em An interior point-proximal method of
  multipliers for convex quadratic programming}, Computational Optimization and
  Applications, 78 (2021), pp.~307--351,
  \url{https://doi.org/10.1002/nla.2361}.

\bibitem{2022IPPMM-SDP}
{\sc S.~Pougkakiotis and J.~Gondzio}, {\em An {{Interior Point-Proximal
  Method}} of {{Multipliers}} for {{Linear Positive Semi-Definite
  Programming}}}, Journal of Optimization Theory and Applications, 192 (2022),
  pp.~97--129, \url{https://doi.org/10.1007/s10957-021-01954-4}.

\bibitem{powell1969method}
{\sc M.~J.~D. Powell}, {\em A method for nonlinear constraints in minimization
  problems}, in Optimization, R.~Fletcher, ed., Academic Press, London, 1969,
  pp.~283--298.

\bibitem{rockafellar1976augmented}
{\sc R.~T. Rockafellar}, {\em Augmented {{Lagrangians}} and {{Applications}} of
  the {{Proximal Point Algorithm}} in {{Convex Programming}}}, Mathematics of
  Operations Research, 1 (1976), pp.~97--116,
  \url{https://doi.org/10.1287/moor.1.2.97}.

\bibitem{santos2019optimized}
{\sc L.-R. Santos, F.~R. {Villas-B{\^o}as}, A.~R.~L. Oliveira, and C.~Perin},
  {\em Optimized choice of parameters in interior-point methods for linear
  programming}, Computational Optimization and Applications, 73 (2019),
  pp.~535--574, \url{https://doi.org/10.1007/s10589-019-00079-9}.

\bibitem{saunders1996cholesky}
{\sc M.~A. Saunders}, {\em Cholesky-based {{Methods}} for {{Sparse Least
  Squares}}: {{The Benefits}} of {{Regularization}}}, in Linear and {{Nonlinear
  Conjugate Gradient-Related Methods}}, L.~Adams and J.~L. Nazareth, eds.,
  SIAM, Philadelphia, 1996, pp.~92--100.

\bibitem{saunders1996solving}
{\sc M.~A. Saunders and J.~A. Tomlin}, {\em Solving regularized linear programs
  using barrier methods and {{KKT}} systems}, Tech. Report IBM Research Report
  RJ 10064 and Stanford SOL Report 96-4, IBM Thomas J. Watson Research Center,
  1996.

\bibitem{schork2020implementation}
{\sc L.~Schork and J.~Gondzio}, {\em Implementation of an interior point method
  with basis preconditioning}, Mathematical Programming Computation, 12 (2020),
  pp.~603--635, \url{https://doi.org/10.1007/s12532-020-00181-8}.

\bibitem{Udell_2019_datalowrank}
{\sc M.~Udell and A.~Townsend}, {\em Why {{Are Big Data Matrices Approximately
  Low Rank}}?}, SIAM Journal on Mathematics of Data Science, 1 (2019),
  pp.~144--160, \url{https://doi.org/10.1137/18M1183480}.

\bibitem{vandenberghe1996semidefinite}
{\sc L.~Vandenberghe and S.~Boyd}, {\em Semidefinite {{Programming}}}, SIAM
  Review, 38 (1996), pp.~49--95, \url{https://doi.org/10.1137/1038003}.

\bibitem{woodsend2009hybrid}
{\sc K.~Woodsend and J.~Gondzio}, {\em Hybrid {{MPI}}/{{OpenMP Parallel Linear
  Support Vector Machine Training}}}, Journal of Machine Learning Research, 10
  (2009), pp.~1937--1953, \url{http://jmlr.org/papers/v10/woodsend09a.html}.

\bibitem{woodsend2011exploiting}
{\sc K.~Woodsend and J.~Gondzio}, {\em Exploiting separability in large-scale
  linear support vector machine training}, Computational Optimization and
  Applications, 49 (2011), pp.~241--269,
  \url{https://doi.org/10.1007/s10589-009-9296-8}.

\bibitem{ye1989extension}
{\sc Y.~Ye and E.~Tse}, {\em An extension of {{Karmarkar}}'s projective
  algorithm for convex quadratic programming}, Mathematical Programming, 44
  (1989), pp.~157--179, \url{https://doi.org/10.1007/BF01587086}.

\bibitem{zhao2022nysadmm}
{\sc S.~Zhao, Z.~Frangella, and M.~Udell}, {\em {{NysADMM}}: Faster composite
  convex optimization via low-rank approximation}, in Proceedings of the 39th
  {{International Conference}} on {{Machine Learning}}, PMLR, June 2022,
  pp.~26824--26840, \url{https://proceedings.mlr.press/v162/zhao22a.html}.

\end{thebibliography}

\appendix
\section{Proofs}

\subsection{Lemmas for \Cref{thm:inexact-IPPMM-QP}} \label{sec:lemmas}

This section provides the lemmas su\-pport\-ing \Cref{thm:inexact-IPPMM-QP}, which is built upon \cite[Lemma 1--4]{PougkakiotisGondzio_2021_IPPMM}, \Cref{lem:bdd-step} and \Cref{lem:IPPMM-stepsize}. 
Note that \cite[Lemma 1--4]{PougkakiotisGondzio_2021_IPPMM} were originally devised for exact IP-PMM, yet their proofs still hold for inexact IP-PMM, provided that the iterates $(x^{k}, y^{k}, z^{k})$ 
lie within the neighborhood $\mathcal{N}_{\mu_k}\left(\zeta^{k}, \lambda^{k}\right)$ defined in \eqref{eq:neighborhood}. 
This condition is established in \Cref{lem:bdd-step,lem:IPPMM-stepsize}, adapted from \cite[Lemma 5--6]{PougkakiotisGondzio_2021_IPPMM} respectively. 
The modifications are inspired by the techniques used for inexact IP-PMM on linearly constrained SDP \cite{2022IPPMM-SDP}.

\begin{lemma}
\label{lem:bdd-step}
Assume \Cref{assump:bdd-opt-sol,assump:constraint-mat-full-rank} and suppose that the search di\-rec\-tion $\left(\Delta x^{k}, \Delta y^{k}, \Delta z^{k}\right)$ satisfies inexact Newton system \eqref{eq:inexact-Newton-RHS} at an arbitrary iteration $k \geq 0$ of inexact IP-PMM.
Let $D_k^2:=\diag(x^{k}) \diag(z^{k})^{-1}$. Then
\begin{equation*}
\left\| D_k^{-1} \Delta x^{k}\right\|_2=O(n^2 \mu_k^{\frac{1}{2}}),
~\left\|D_k \Delta z^{k}\right\|_2=O(n^2 \mu_k^{\frac{1}{2}}), 
~\left\|\left(\Delta x^{k}, \Delta y^{k}, \Delta z^{k}\right)\right\|_2=O(n^3).
\end{equation*}
\end{lemma}
\begin{proof}
The proof follows that of \cite[Lemma 5]{PougkakiotisGondzio_2021_IPPMM} with a few modifications. 
For simplicity, we only sketch the proof with emphasis on modifications.

Following the beginning of the proof for \cite[Lemma 5]{PougkakiotisGondzio_2021_IPPMM}, 
we invoke \cite[Lemma 2]{PougkakiotisGondzio_2021_IPPMM} and \cite[Lemma 3]{PougkakiotisGondzio_2021_IPPMM} to obtain two triples $(x_{r^{k}}^{\ast}, y_{r^{k}}^{\ast}, z_{r^{k}}^{\ast})$ and $(\tilde{x}, \tilde{y}, \tilde{z})$ satisfying \cite[Eq. (22)]{PougkakiotisGondzio_2021_IPPMM} and \cite[Eq. (23)]{PougkakiotisGondzio_2021_IPPMM} respectively.
Using the centering parameter $\sigma_k$, instead of defining $\hat{c}$ and $\hat{b}$ as in \cite[Eq. (29)]{PougkakiotisGondzio_2021_IPPMM}, we define
\begin{align}
\hat{c} & :=-\left(\sigma_k \frac{\bar{c}}{\mu_0}-\left(1-\sigma_k\right)\left(x^{k}-\zeta^{k}+\frac{\mu_k}{\mu_0}\left(\tilde{x}-x_{r^{k}}^*\right)\right) + \frac{1}{\mu_k} \inerror_{d}^{k} \right), \label{eq:pf-b-hat} \\
\hat{b} & :=-\left(\sigma_k \frac{\bar{b}}{\mu_0}+\left(1-\sigma_k\right)\left(y^{k}-\lambda^{k}+\frac{\mu_k}{\mu_0}\left(\tilde{y}-y_{r^{k}}^*\right)\right) + \frac{1}{\mu_k} \inerror_{p}^{k} \right), \label{eq:pf-c-hat}
\end{align}
where $\bar{c}, \bar{b}, \mu_0$ are defined in \eqref{eq:starting-pt}.
The additional terms $\inerror_d^{k} / \mu_k$ and $\inerror_p^{k} / \mu_k$ in \eqref{eq:pf-b-hat}--\eqref{eq:pf-c-hat} take into account the inexact errors in Newton system \eqref{eq:inexact-Newton-RHS}. 
With this new pair $(\hat{b}, \hat{c})$, define the triple $(\bar{x}, \bar{y}, \bar{z})$ as in \cite[Eq. (31)]{PougkakiotisGondzio_2021_IPPMM}.
Using \cite[Eq. (22)--(23), (31)]{PougkakiotisGondzio_2021_IPPMM} and inexact Newton system \eqref{eq:inexact-Newton-RHS}, it can be directly verified that
\begin{equation} \label{eq:pf-lem5-1}
A \bar{x} + \mu_k \bar{y} = 0, \quad -Q \bar{x} + A^T \bar{y} + \bar{z} - \mu_k \bar{x} = 0.
\end{equation}
Additional error terms are absorbed into the inexactness of the Newton system \eqref{eq:inexact-Newton-RHS}, yielding \eqref{eq:pf-lem5-1}. 
Given \eqref{eq:pf-lem5-1}, the rest of the proof 
follows the proof of \cite[Lemma 5]{PougkakiotisGondzio_2021_IPPMM}.
\end{proof}

\Cref{lem:IPPMM-stepsize} serves as the keystone for \Cref{thm:inexact-IPPMM-QP}, establishing a connection between inexact and exact IP-PMM. The lemma demonstrates that a step in the \textit{inexact} search direction with a sufficiently small stepsize yields sufficient reduction in the duality measure (\Cref{lem:IPPMM-stepsize} \eqref{item:lem-stepsize-a}) while staying within neighborhood $\mathcal{N}_{\mu_k}\left(\zeta^{k}, \lambda^{k}\right)$ (\Cref{lem:IPPMM-stepsize} \eqref{item:lem-stepsize-c}). Additionally, it ensures that the stepsizes are uniformly bounded away from zero (for every iteration), on the order of $O(\frac{1}{n^4})$ (\Cref{lem:IPPMM-stepsize} \eqref{item:lem-stepsize-b}).

\begin{lemma}
\label{lem:IPPMM-stepsize}
Instate \Cref{assump:bdd-opt-sol,assump:constraint-mat-full-rank} and assume that the search di\-rec\-tion $(\Delta x^{k}, \Delta y^{k}, \Delta z^{k})$ satisfies the inexact Newton system \eqref{eq:inexact-Newton-RHS}, with inexact errors satisfying \Cref{assump:Newton-error-bound} for every iteration $k \geq 0$ of inexact IP-PMM. 
Then there exists a step-length $\bar{\alpha} \in(0,1)$ such that the following hold for all iterations $k \geq 0$:
\begin{enumerate}[(i)]
\item \label{item:lem-stepsize-a} 
for all $\alpha \in[0, \bar{\alpha}]$,
\begin{align}
&\left(x^{k}+\alpha \Delta x^{k}\right)^T\left(z^{k}+\alpha \Delta z^{k}\right) \geq \left(1-\alpha\left(1-\frac{\sigma_{\min}}{2} \right)\right) (x^{k})^T z^{k}; \label{eq:lem-stepsize-mu-1} \\
&\left(x^{k}_i+\alpha \Delta x^{k}_i\right)\left(z^{k}_i+\alpha \Delta z^{k}_i\right) \geq \frac{\gamma_\mu}{n}\left(x^{k}+\alpha \Delta x^{k}\right)^T\left(z^{k}+\alpha \Delta z^{k}\right), ~\forall i \in [n]; \label{eq:lem-stepsize-mu-2} \\
&\left(x^{k}+\alpha \Delta x^{k}\right)^T\left(z^{k}+\alpha \Delta z^{k}\right) \leq\left(1 - 0.01 \alpha \right) (x^{k})^T z^{k}; \label{eq:lem-stepsize-mu-3}
\end{align}

\item \label{item:lem-stepsize-b} 
$\bar{\alpha} \geq \bar{\kappa}/n^4$ with $\bar{\kappa}>0$ being a constant independent of $n$ and $m$;

\item \label{item:lem-stepsize-c} 
suppose $\left(x^{k}, y^{k}, z^{k}\right) \in \mathcal{N}_{\mu_k}\left(\zeta^{k}, \lambda^{k}\right)$ and, for any $\alpha \in(0, \bar{\alpha}]$, let 
\begin{equation} \label{eq:k+1-iterate}
\left(x^{k+1}, y^{k+1}, z^{k+1}\right) =\left(x^{k}+\alpha \Delta x^{k}, y^{k}+\alpha \Delta y^{k}, z^{k}+\alpha \Delta z^{k}\right)
\end{equation}
and $\mu_{k+1} = (x^{k+1})^T z^{k+1} / n$.
Then 
\begin{equation} \label{eq:k+1-in-neighborhood}
\left(x^{k+1}, y^{k+1}, z^{k+1}\right) \in \mathcal{N}_{\mu_{k+1}}\left(\zeta^{k+1}, \lambda^{k+1}\right),
\end{equation}
where $\lambda^{k}$ and $\zeta^{k}$ are updated as in \Cref{alg:inexact-IPPMM}.
\end{enumerate}
\end{lemma}
\begin{proof} 
The proof follows that of \cite[Lemma 6]{PougkakiotisGondzio_2021_IPPMM}.
Again, we only sketch the outline of the proof and point out the necessary modifications.
The assumption $\inerror_{\mu}^{k} = 0$ in \Cref{assump:Newton-error-bound} guarantees the third block of Newton system \eqref{eq:inexact-Newton-RHS} is exact, 
so same argument in the proof of \cite[Lemma 6]{PougkakiotisGondzio_2021_IPPMM} guarantee that \eqref{eq:lem-stepsize-mu-1}--\eqref{eq:lem-stepsize-mu-3} hold for every $\alpha \in (0, \alpha^{\ast})$ with exactly the same $\alpha^{\ast}$ in \cite[Eq. (40)]{PougkakiotisGondzio_2021_IPPMM}, \ie
\begin{equation} \label{eq:alpha-ast}
\alpha^* \coloneqq \min \left\{\frac{\sigma_{\min }}{2 C_{\Delta}^2 n^3}, \frac{\sigma_{\min }\left(1-\gamma_\mu\right)}{2 C_{\Delta}^2 n^4}, \frac{0.49}{C_{\Delta}^2 n^3}, 1\right\},
\end{equation}
where $C_{\Delta} > 0$ is a constant such that $(\Delta x^{k})^T \Delta z^{k} \leq \|D_k^{-1} \Delta x^{k}\|_2 \|D_k \Delta z^{k}\|_2 \leq C_{\Delta}^2 n^4 \mu_k$ from \Cref{lem:bdd-step}.
Next, we must find the maximal $\bar{\alpha} \in (0, \alpha^{\ast}]$ such that 
\begin{equation} \label{eq:pf-in-neighborhood}
\left(x^{k}(\alpha), y^{k}(\alpha), z^{k}(\alpha)\right) \in \mathcal{N}_{\mu_k(\alpha)}(\zeta^{k}, \lambda^{k}), \quad \text{for all } \alpha \in (0, \bar{\alpha}],
\end{equation}
where $\mu_{k}(\alpha) = \frac{x(\alpha)^T z(\alpha)}{n}$ and
\begin{equation*}
\left(x^{k}(\alpha), y^{k}(\alpha), z^{k}(\alpha)\right) = 
\left(x^{k} + \alpha \Delta x^{k}, y^{k} + \alpha \Delta y^{k}, z^{k} + \alpha \Delta z^{k}\right).
\end{equation*}
By definitions \eqref{eq:C-tilde-set} and \eqref{eq:neighborhood}, \cref{eq:pf-in-neighborhood} is equivalent to
\begin{equation} \label{eq:pf-r-tilde-goal}
\left\|\tilde{r}_p(\alpha), \tilde{r}_d(\alpha)\right\|_2 \leq C_N \tfrac{\mu_k(\alpha)}{\mu_0}, \quad\left\|\tilde{r}_p(\alpha), \tilde{r}_d(\alpha)\right\|_{\mathcal{A}} \leq \gamma_{\mathcal{A}} \rho \tfrac{\mu_k(\alpha)}{\mu_0}, \quad \forall \alpha \in(0, \bar{\alpha}],
\end{equation}
where $\tilde{r}_p(\alpha)$ and $\tilde{r}_d(\alpha)$ are the residuals of two equalities in $\tilde{\mathcal{C}}_{\mu_k}\left( \zeta^{k}, \lambda^{k} \right)$ in \eqref{eq:C-tilde-set}: 
\begin{align*}
\tilde{r}_p(\alpha) &= A x^{k}(\alpha)+\mu_k(\alpha)\left(y^{k}(\alpha)-\lambda^{k}\right)-\left(b+\frac{\mu_k(\alpha)}{\mu_0} \bar{b}\right), \\
\tilde{r}_d(\alpha) &= -Q x^{k}(\alpha)+A^T y^{k}(\alpha)+z^{k}(\alpha)-\mu_k(\alpha)\left(x^{k}(\alpha)-\zeta^{k}\right)-\left(c+\frac{\mu_k(\alpha)}{\mu_0} \bar{c}\right).
\end{align*}
Following the calculations in the proof of \cite[Lemma 6]{PougkakiotisGondzio_2021_IPPMM}, now with the inexact Newton system \eqref{eq:inexact-Newton-RHS}, the two residuals $\tilde{r}_p(\alpha)$ and $\tilde{r}_d(\alpha)$ have additional inexact error terms $\alpha \inerror_{p}^{k}$ and $\alpha \inerror_{d}^{k}$ respectively:
{\small
\begin{align*}
\tilde{r}_p(\alpha) &= (1-\alpha) \frac{\mu_k}{\mu_0} \tilde{b}^{k} + \alpha^2\left(\sigma_k-1\right) \mu_k \Delta y^{k}
+\alpha^2 \tfrac{(\Delta x^{k})^T \Delta z^{k}}{n}\left(y^{k}-\lambda^{k}+\alpha \Delta y^{k}-\frac{\bar{b}}{\mu_0}\right) + \alpha \inerror_{p}^{k}, \\
\tilde{r}_d(\alpha) &= (1-\alpha) \frac{\mu_k}{\mu_0} \tilde{c}^{k} - \alpha^2\left(\sigma_k-1\right) \mu_k \Delta x^{k} 
-\alpha^2 \tfrac{(\Delta x^{k})^T \Delta z^{k}}{n}\left(x^{k}-\zeta^{k}+\alpha \Delta x^{k}+\frac{\bar{c}}{\mu_0}\right) + \alpha \inerror_{d}^{k}.
\end{align*}
}
Using the same quantities $\xi_2 = O(n^4 \mu_k)$ and $\xi_{\mathcal{A}} = O(n^4 \mu_k)$ defined as in \cite[Eq. (44)]{PougkakiotisGondzio_2021_IPPMM}, we have the bounds
\begin{align*}
\left\|(\tilde{r}_p(\alpha), \tilde{r}_d(\alpha))\right\|_2 &\leq (1-\alpha) C_N \frac{\mu_k}{\mu_0}+\alpha^2 \mu_k \xi_2 + \alpha \|(\inerror_{p}^{k}, \inerror_{d}^{k})\|_2, \\
\left\|(\tilde{r}_p(\alpha), \tilde{r}_d(\alpha))\right\|_{\mathcal{A}} &\leq (1-\alpha) \gamma_{\mathcal{A}} \rho \frac{\mu_k}{\mu_0}+\alpha^2 \mu_k \xi_{\mathcal{A}} + \alpha \|(\inerror_{p}^{k}, \inerror_{d}^{k})\|_{\mathcal{A}},
\end{align*}
for all $\alpha \in (0, \alpha^{\ast}]$, where $\alpha^{\ast}$ is given by \eqref{eq:alpha-ast}.
\Cref{assump:Newton-error-bound} on error bounds for the inexactness further yields that, for all $\alpha \in (0, \alpha^{\ast}]$, 
\begin{align}
\left\|\tilde{r}_p(\alpha), \tilde{r}_d(\alpha)\right\|_2 
&\leq
\frac{C_N}{\mu_0} \left(1 - \alpha \left(1 - \frac{\sigma_{\min}}{4} - \alpha \frac{\mu_0 \xi_2}{C_N} \right) \right) \mu_k, \label{eq:pf-r-tilde-1} \\
\left\|\tilde{r}_p(\alpha), \tilde{r}_d(\alpha)\right\|_{\mathcal{A}}
&\leq
\frac{\gamma_{\mathcal{A}} \rho}{\mu_0} \left(1 - \alpha \left(1 - \frac{\sigma_{\min}}{4} - \alpha \frac{\mu_0 \xi_{\mathcal{A}}}{\gamma_{\mathcal{A}} \rho} \right) \right) \mu_k. \label{eq:pf-r-tilde-2}
\end{align}
On the other hand, \eqref{eq:lem-stepsize-mu-1} implies $\mu_k(\alpha) \geq \left(1-\alpha\left(1-\frac{\sigma_{\min}}{2} \right)\right) \mu_k$ for all $\alpha \in (0, \alpha^{\ast}]$, which together with \eqref{eq:pf-r-tilde-1}--\eqref{eq:pf-r-tilde-2} yields
\begin{alignat*}{2}
\left\|\tilde{r}_p(\alpha), \tilde{r}_d(\alpha)\right\|_2 &\leq \frac{\mu_k(\alpha)}{\mu_0} C_N, \quad &&\text{for all } \alpha \in\left(0, \min \left\{\alpha^*, \frac{\sigma_{\min} C_N}{4 \xi_2 \mu_0}\right\}\right]; \\
\left\|\tilde{r}_p(\alpha), \tilde{r}_d(\alpha)\right\|_{\mathcal{A}} &\leq \frac{\mu_k(\alpha)}{\mu_0} \gamma_{\mathcal{A}} \rho, \quad &&\text{for all } \alpha \in\left(0, \min \left\{\alpha^*, \frac{\sigma_{\min} \gamma_{\mathcal{A}} \rho}{4 \xi_{\mathcal{A}} \mu_0}\right\}\right] .
\end{alignat*}
Finally, let
\begin{equation} \label{eq:alpha-bar}  
\bar{\alpha}^{k} \coloneqq \min \left\{\alpha^*, \frac{\sigma_{\min} C_N}{4 \xi_2 \mu_0}, \frac{\sigma_{\min} \gamma_{\mathcal{A}} \rho}{4 \xi_{\mathcal{A}} \mu_0} \right\} 
~\text{ and }~
\bar{\alpha} \coloneqq \min_{k \geq 0} \bar{\alpha}^{k}.
\end{equation} 
Since $\xi_2 = O(n^4 \mu_k)$, $\xi_{\mathcal{A}} = O(n^4 \mu_k)$, and $\mu_k$ is decreasing, there must exist a constant $\bar{\kappa}>0$, independent of $n, m$, and $k$ such that $\bar{\alpha}^{k} \geq \frac{\bar{\kappa}}{n^4}$ for all $k \geq 0$, and hence $\bar{\alpha} \geq \frac{\bar{\kappa}}{n^4}$. 
This proves \eqref{item:lem-stepsize-a} and \eqref{item:lem-stepsize-b}.
To prove \eqref{item:lem-stepsize-c}, we consider two cases:
\begin{itemize}
\item Suppose the estimates $\zeta^{k}$ and $\lambda^{k}$ are not updated, that is, $\zeta^{k+1} = \zeta^{k}$ and $\lambda^{k+1} = \lambda^{k}$.
The choice of $\bar{\alpha}$ in \eqref{eq:alpha-bar} guarantees \eqref{eq:pf-r-tilde-goal} and hence \eqref{eq:pf-in-neighborhood}, so that the new iterate in \eqref{eq:k+1-iterate} satisfies \eqref{eq:k+1-in-neighborhood}.

\item Suppose IP-PMM updates the estimates $\zeta^{k}$ and $\lambda^{k}$, that is, $\zeta^{k+1} = x^{k+1}$ and $\lambda^{k+1} = y^{k+1}$. This happens only when 
\begin{equation} \label{eq:update-cond}
\left\|(r_p, r_d) \right\|_2 \leq C_N \frac{\mu_{k+1}}{\mu_0}, \quad \left\|(r_p,r_d)\right\|_{\mathcal{A}} \leq \gamma_{\mathcal{A}} \rho \frac{\mu_{k+1}}{\mu_0},
\end{equation}
where $r_p = A x^{k+1} - (b + \frac{\mu_{k+1}}{\mu_0} \bar{b})$ and $r_d = (c + \frac{\mu_{k+1}}{\mu_0} \bar{c}) + Q x^{k+1} - A^T y^{k+1} - z^{k+1}$.
Condition \eqref{eq:update-cond} is exactly equivalent to \eqref{eq:k+1-in-neighborhood}.
\end{itemize}
\end{proof}

\subsection{Proof of \Cref{prop:errors-propagation}} 
\label{sec:pf-prop}
\begin{proof}
We omit the proof for $(\inerror_{d}^{k}, \inerror_{p}^{k}, \inerror_{\mu}^{k}) = (0, \inerrornormeq^{k}, 0)$ as it follows from direct calculations.
The equality and bounds for $2$-norm in \eqref{eq:prop-error-goal} follow directly from the fact $(\inerror_{d}^{k}, \inerror_{p}^{k}, \inerror_{\mu}^{k}) = (0, \inerrornormeq^{k}, 0)$ and the assumption $\| \inerrornormeq^{k} \|_2 \leq C \mu_k$.
To verify bound for semi-norm $\|(\inerror_{p}^{k}, \inerror_{d}^{k})\|_{\mathcal{A}}$, 
observe that 
\begin{equation*}
\|(\inerror_{p}^{k}, \inerror_{d}^{k})\|_{\mathcal{A}} 
= \|(0, \epsilon)\|_{\mathcal{\mathcal{A}}}
= \min_{x, y, z} \left\{ \|(x,z)\|_2 \mid Ax = 0, ~ -Q x + A^T y + z = \epsilon \right\},
\end{equation*}
and $(\tilde{x}, \tilde{y}, \tilde{z}) = (0, 0, \inerrornormeq^{k})$ is a solution to the systems $Ax = 0$ and $-Q x + A^T y + z = \inerrornormeq^{k}$, so that $\|(\inerror_{p}^{k}, \inerror_{d}^{k})\|_{\mathcal{\mathcal{A}}}  
\leq \left\|\left(\tilde{x}, \tilde{z}\right)\right\|_2 = \|(0, \inerrornormeq^{k})\|_2
\leq C \mu_k$.
\end{proof}

\subsection{Proof of \Cref{thm:main} and required lemmas} \label{sec:pf-Nys-IP-PMM-convergence}
This section presents the proof of \Cref{thm:main} and the necessary lemmas. 
\Cref{lem:NysPCG-relative-error} from \cite{zhao2022nysadmm} gives the number of Nystr\"om PCG iterations to obtain a solution within error $\varepsilon$.

\begin{lemma}[{\cite[Corollary 4.2]{zhao2022nysadmm}}] \label{lem:NysPCG-relative-error}
Let $\delta > 0$ and consider regularized system
\begin{equation} \label{eq:reg-normal-mu}
(N + \delta I) \Delta y = \xi.
\end{equation}
Suppose $P$ is the Nystr\"om preconditioner constructed as in \eqref{eq:Nys_precond} and sketch size 
$\ell \geq 8 \left( \sqrt{d_{\text{eff}}(N, \delta)} + \sqrt{8 \log(16/\eta)} \right)^2.$ Let $\Delta y^{\ast}$ denote the true solution of \eqref{eq:reg-normal-mu} and $\{\Delta y^{(t)}\}_{t \geq 1}$ denote the iterates generated by Nystr\"om PCG on problem \eqref{eq:reg-normal-mu}.
Then with probability at least $1 - \eta$, it holds true that $\|\Delta y^{(t)} - \Delta y^{\star}\|_2 \leq \varepsilon$ after $t = O(\log(\|\Delta y^{\star}\|_2 / \varepsilon))$ iterations.
\end{lemma}

Based on \Cref{lem:NysPCG-relative-error}, we present next \Cref{lem:NysPCG-for-IPPMM} showing that at each iteration $k$ of Nys-IP-PMM, Nystr\"om PCG can return a solution with small residual error after a few iterations, when the Newton system \eqref{eq:inexact-Newton-RHS} is solved by the normal equations.
For simplicity of presentation, we drop the iteration superindex $k$ in the statement and proof of the lemma.

\begin{lemma} \label{lem:NysPCG-for-IPPMM}
Let $N = A (Q + \Theta^{-1} + \mu I_n)^{-1} A^T$. 
Given \Cref{assump:bdd-opt-sol,assump:constraint-mat-full-rank}, suppose the regularized normal equations \eqref{eq:normal-eq-inexact} is solved by PCG using the randomized Nystr\"om preconditioner \eqref{eq:Nys_precond} with $\delta = \mu$ and sketch size 
\begin{equation} \label{eq:lem-NysPCG-sketch-size}
\ell \geq 8 \left( \sqrt{d_{\text{eff}}(N, \mu)} + \sqrt{8 \log(16/\eta)} \right)^2,
\end{equation}
where $d_{\text{eff}}(\cdot)$ is the effective dimension of $N$ defined in \eqref{eq:effective-dim}.
Let $\{\Delta y^{(t)}\}_{t \geq 1}$ denote the iterates generated by Nystr\"om PCG.
Then, with probability at least $1 - \eta$, after $t = O\left(\log(\frac{n}{\varepsilon \mu})\right)$ iterations, the residual satisfies 
\begin{equation} \label{eq:lem-residual-goal}
\left\|(N + \mu I) \Delta y^{(t)} - \xi \right\|_2 \leq \varepsilon.
\end{equation}
\end{lemma}
\begin{proof}
We abbreviate the regularized matrix as $N_{\mu} = N + \mu I$ and let $\Delta y^{\star}$ denote the true solution of \eqref{eq:normal-eq} (with $\rho = \delta = \mu$). 
Let $r^{(t)} = N_{\mu} \Delta y^{(t)} - \xi$ denote the residual at $t$-th iteration.
It follows from norm consistency and $\|N_{\mu}\|_2 = \lambda_{\max}(N_{\mu})$ that $\|r^{(t)}\|_2 \leq \lambda_{\max}(N_{\mu}) \|\Delta y^{(t)} - \Delta y^{\star}\|_2$.
Hence, we would have $\|r^{(t)}\|_2 \leq \varepsilon$ if 
\begin{equation} \label{eq:lem-residual-2}
\|\Delta y^{(t)} - \Delta y^{\star}\|_2 \leq \varepsilon/\lambda_{\max}(N_{\mu}).
\end{equation}
By \Cref{lem:NysPCG-relative-error}, with probability at least $1 - \eta$, inequality \eqref{eq:lem-residual-2} holds true after  
\begin{equation} \label{eq:lem-residual-3}
t = O\left( \log\left(\frac{ \lambda_{\max}(N_{\mu})  \| \Delta y^{\star} \|_2}{\varepsilon}\right) \right)
\end{equation}
PCG iterations. 
\Cref{assump:constraint-mat-full-rank} and \Cref{lem:bdd-step} respectively guarantee that
\begin{align*}
\lambda_{\max}(N_{\mu}) 
= \left\| A (Q + \Theta + \mu I)^{-1} A^T \right\|_2 + \mu
= O\left( \frac{1}{\mu} \right) \text{ and } \left\|\Delta y^{\star} \right\|_{2}
= O\left( n^3 \right).
\end{align*}
Therefore, the required number of iterations \eqref{eq:lem-residual-3} becomes $t = O\left( \log( \frac{n}{\varepsilon \mu} ) \right).$
\end{proof}

\begin{proof}[Proof of \Cref{thm:main}]
To guarantee the convergence, the inexact errors of IP-PMM have to satisfy \Cref{assump:Newton-error-bound}, which holds true if, by \Cref{prop:errors-propagation}, the inexact error $\inerrornormeq^{k}$ in \eqref{eq:normal-eq-inexact} satisfies
\begin{equation} \label{eq:alg-Nys-IP-PMM-error}
\| \inerrornormeq^{k} \|_2 \leq \frac{\sigma_{\min}}{4 \mu_0} \min\left\{ C_N, \gamma_{\mathcal{A}} \rho \right\} \mu_k.
\end{equation}
At each iteration $k$ of Nys-IP-PMM, we take the sketch size $\ell_k$ as in \eqref{eq:main-thm-sketch-size} and $\eta_k \coloneqq \frac{1}{(k+2)^4}$ for all $k \geq 0$, so that \eqref{eq:lem-NysPCG-sketch-size} in \Cref{lem:NysPCG-for-IPPMM} is satisfied, which guarantees, with probability at least $1 - \eta_k$, that Nystr\"om PCG satisfies \eqref{eq:alg-Nys-IP-PMM-error} with order of iterations given by $O(\log \frac{n}{C \mu_k^2}) = O(\log (n / \varepsilon))$, where $C \coloneqq \frac{\sigma_{\min}}{4 \mu_0} \min\left\{ C_N, \gamma_{\mathcal{A}} \rho \right\}$.
To get this, we use the fact that $\frac{1}{\mu_k} = O(\frac{1}{\varepsilon})$ before Nys-IP-PMM terminates.
By intersecting all these events across iteration count $k$, we conclude: with probability at least $1 - \sum_{k = 0}^{\infty} \frac{1}{(k+2)^4} \approx 0.91$.
Nystr\"om PCG terminates within $t = O(\log \frac{n}{\varepsilon})$ steps and the residual error of PCG satisfies \eqref{eq:alg-Nys-IP-PMM-error} for all $k \geq 0$, which proves \eqref{item:thm-main-i}. 

\Cref{prop:errors-propagation} then guarantees that $(\Delta x^{k}, \Delta y^{k}, \Delta z^{k})$ satisfies inexact Newton system \eqref{eq:inexact-Newton-RHS} with inexact errors satisfying \Cref{assump:Newton-error-bound}.
Now, we are under the assumptions of \Cref{thm:inexact-IPPMM-QP}, so we have $\mu_k \leq \varepsilon$ after $k = O(n^4 \log \frac{1}{\varepsilon})$ iterations of Nys-IP-PMM, as required in \eqref{item:thm-main-ii}.
\end{proof}
\section{Details on implementation of Nys-IP-PMM}

Nys-IP-PMM deviates from the theory in order to improve the computational efficiency, following  \cite{PougkakiotisGondzio_2021_IPPMM} in the following two main aspects: 
First, we do not limit to $\rho_k = \delta_k = \mu_k$ as theory does. 
The proximal parameters $\rho_k$ and $\delta_k$ are independently updated following the suggestions from \cite[Algorithm PEU]{PougkakiotisGondzio_2021_IPPMM}. 
Second, the iterates of the method are not required to lie in the neighborhood as in \eqref{eq:neighborhood} for efficiency.

The followings provide more details/explanations for the implementation of Nys-IP-PMM. 
\Cref{sec:derivation-free-box-newton} de\-rives the Newton system to be solved at each Nys-IP-PMM iteration for QP instance taking the form $(\tilde{\text{P}})$--$(\tilde{\text{D}})$ and the resulting normal equations.
\Cref{sec:initial-pt} details the construction of practical initial point.

\subsection{Derivations of Newton system and equations \eqref{eq:normal-eq-free-box}--\eqref{eq:delta-s-free-box}}
\label{sec:derivation-free-box-newton}

Given QP in the form $(\tilde{\text{P}})$--$(\tilde{\text{D}})$, IP-PMM solves a sequence of subproblems taking the form \eqref{eq:PMM-subproblem-free-box}. 
Introducing the logarithmic barrier function to enforce the constraints $x^{\I} \geq 0$ and $0 \leq x^{\J} \leq u^{\J}$ in \eqref{eq:PMM-subproblem-free-box}, the Lagrangian to minimize becomes
\begin{align*}
\mathcal{L}(x) = \frac{1}{2} x^T Q x + c^T x + (\lambda^{k})^T (b - Ax) &+ \frac{1}{2 \delta_k} \|Ax - b\|_2^2 + \frac{\rho_k}{2} \|x - \zeta^{k}\|^2 \\
&- \mu_k \sum_{j \in \I\J} \ln x_{j} - \mu_k \sum_{j \in \J} \ln(u_{j} - x_{j}).
\end{align*}
Setting $\nabla_x \mathcal{L}(x) = 0$ and introducing the new variables
\begin{alignat*}{2}
&y = \lambda^{k} - \frac{1}{\mu_k}(Ax-b), \quad
&&z_{j} = 
\begin{cases}
\mu_k x_j^{-1}, &\text{if } j \in \I\J; \\
0, &\text{if } j \in \F,
\end{cases} 
\\
&w_{j} = 
\begin{cases}
u_{j} - x_j, &\text{if } j \in \J; \\
0, &\text{if } j \in \I\F,
\end{cases} \quad 
&&s_{j} = 
\begin{cases}
\mu_k w_j^{-1}, &\text{if } j \in \J; \\
0, &\text{if } j \in \I\F,
\end{cases}
\end{alignat*}
we obtain the following (non-linear) system to solve:
\begin{equation} \label{eq:upper-bound-F=0}
\begin{bmatrix}
c + Qx - A^T y -z + s + \rho_k (x - \zeta^{k}) \\
Ax - b + \delta_k(y - \lambda^{k}) \\
x^{\J} + w^{\J} - u^{\J} \\
\diag(x^{\I\J}) z^{\I\J} - \mu_k \mathbbm{1}_{|\I\J|} \\
\diag(w^{\J}) s^{\J} - \mu_k \mathbbm{1}_{|\J|}
\end{bmatrix}
=
\begin{bmatrix}
0 \\ 0 \\ 0 \\ 0 \\ 0
\end{bmatrix}.
\end{equation}
Applying Newton's method to \eqref{eq:upper-bound-F=0} yields the Newton system taking the form:
\begin{equation} \label{eq:Newton-free-box}
\begin{bmatrix}
-\left(Q+\rho_k I_n\right) & A^T & 0 & I_n & -I_n \\
A & \delta_k I_m & 0 & 0 & 0 \\
E_{\J} & 0 & E_{\J} & 0 & 0 \\
E_{\I\J}Z_k & 0 & 0 & E_{\I\J}X_k & 0 \\
0 & 0 & E_{\J}S_k & 0 & E_{\J}W_k
\end{bmatrix}\begin{bmatrix}
\Delta x^{k} \\
\Delta y^{k} \\
\Delta w^{k} \\
\Delta z^{k} \\
\Delta s^{k}
\end{bmatrix}
=
\begin{bmatrix}
r_{d} \\
r_{p} \\
(r_{u})_{\J} \\
(r_{xz})_{\I\J} \\
(r_{ws})_{\J}
\end{bmatrix}
\end{equation}
where $E_{\mathcal{S}}: \R^{n} \rightarrow \R^{|\mathcal{S}|}$ is the projection matrix such that $E_{\mathcal{S}} x = x^{\mathcal{S}}$ for all $x \in \R^n$, the uppercase letters $X^k, Z^k, W^k, S^k$ represent diagonal matrices with diagonal entries corresponding to the iterates in the associated lowercase letters, and $r_d, r_p, r_u, r_{xz}, r_{ws}$ are appropriate RHS vectors.
By eliminating the variables $\Delta x^{k}, \Delta w^{k}, \Delta z^{k},$ and $\Delta s^{k}$, the Newton system \eqref{eq:Newton-free-box} simplifies to the normal equations \eqref{eq:normal-eq-free-box} of size $m$, and four closed-form formulae \eqref{eq:delta-x-free-box}--\eqref{eq:delta-s-free-box}.

\subsection{Construction for initial point} \label{sec:initial-pt}
The construction is based on the development of initial point for original IP-PMM in \cite{PougkakiotisGondzio_2021_IPPMM}. 
We first construct a candidate point by ignoring the non-negative constraints and solving the primal and dual equality constraints $Ax = b$ and $-Qx + A^T y + z - s = c$ in $(\tilde{\text{P}})$--$(\tilde{\text{D}})$.
We force primal candidate $\tilde{x}$ satisfy $Ax = b$ while centering around $u / 2$, and solve the dual candidate $\tilde{y}$ from the least squares problem $\min_{y} \|-Q \tilde{x} + A^T y - c\|_2$, ignoring $z$ and $s$ in dual constraint meanwhile.
Then dual candidates $\tilde{z}$ and $\tilde{s}$ are chosen such that dual constraint $-Q \tilde{x} + A^T \tilde{y} + \tilde{z} - \tilde{s} = c$ is satisfied. The primal candidate $\tilde{w}$ is set as $\tilde{w}_{\J} = u_{\J} - \tilde{x}_{\J}$ and $\tilde{w}_{\I\F} = 0$. In summary, the candidate point is:
\begin{align*}
&\tilde{x} = \frac{u}{2} + A^T (A A^T)^{-1} \left(b - \frac{1}{2} Au \right), \quad
\tilde{y} = (A A^T)^{-1} A (c + Q \tilde{x}), \\
&\tilde{z}_{\J} = \frac{1}{2} \left( c - A^T \tilde{y} + Q \tilde{x} \right)_{\J}, \quad \tilde{z}_{\I\F} = \left( c - A^T \tilde{y} + Q \tilde{x} \right)_{\I\F}, \\
&\tilde{s}_{\J} = - \tilde{z}_{\J}, \quad \tilde{s}_{\I\F} = 0, \quad \tilde{w}_{\J} = u_{\J} - \tilde{x}_{\J}, \quad \tilde{w}_{\I\F} = 0.
\end{align*}
However, to ensure the stability and efficiency, we regularize the matrix $A A^T$ as in \cite{PougkakiotisGondzio_2021_IPPMM} and use Nystr\"om PCG to solve for the systems in $\tilde{x}$ and $\tilde{y}$, for which the regularization parameter is taken as $\delta = 10$.

Next, to guarantee the positivity of $x_{\I\J}$, $z_{\I\J}$, $w_{\J}$, and $s_{\J}$, we compute the following quantities:
\begin{align*}
&\delta_{p} = \max\{ -1.5 \min(\tilde{x}_{\I\J}), -1.5 \min(\tilde{w}_{\J}), 0 \}, \\ 
&\delta_{d} = \max\{ -1.5 \min(\tilde{z}_{\I\J}), -1.5 \min(\tilde{s}_{\J}), 0 \}, \\
&\gamma_{xz} = (\tilde{x}_{\I\J} + \delta_{p} \mathbbm{1}_{|\I\J|})^T (\tilde{z}_{\I\J} + \delta_{d} \mathbbm{1}_{|\I\J|}), \\
&\gamma_{ws} = (\tilde{w}_{\J} + \delta_{p} \mathbbm{1}_{|\J|})^T (\tilde{s}_{\J} + \delta_{d} \mathbbm{1}_{|\J|}), \\
&\tilde{\delta}_{p} = \delta_{p} + 0.5 \frac{\gamma_{xz} + \gamma_{ws}}{\sum_{\ell \in \I\J} (\tilde{z}^{\ell} + \delta_{d}) + \sum_{\ell \in \J} (\tilde{s}^{\ell} + \delta_{d})}, \\
&\tilde{\delta}_{d} = \delta_{d} + 0.5 \frac{\gamma_{xz} + \gamma_{ws}}{\sum_{\ell \in \I\J} (\tilde{x}^{\ell} + \delta_{d}) + \sum_{\ell \in \J} (\tilde{w}^{\ell} + \delta_{d})}.
\end{align*}
Finally, we set the final initial point as:
\begin{equation} \label{eq:initial-pt-pratical}  
\begin{aligned}
&y^{0} = \tilde{y}, \quad
x^{0}_{\I\J} = \tilde{x}_{\I\J} + \tilde{\delta}_{p} \mathbbm{1}_{|\I\J|}, \quad x^{0}_{\F} = \tilde{x}_{\F}, \\
&w^{0}_{\J} = \tilde{w}_{\J} + \tilde{\delta}_{p} \mathbbm{1}_{|\J|}, \quad w^{0}_{\I\F} = 0, \\
&z^{0}_{\I\J} = \tilde{z}_{\I\J} + \tilde{\delta}_{d} \mathbbm{1}_{|\I\J|}, \quad z^{0}_{\F} = 0, \quad
s^{0}_{\J} = \tilde{s}_{\J} + \tilde{\delta}_{d} \mathbbm{1}_{|\J|}, \quad s^{0}_{\I\F} = 0.
\end{aligned}
\end{equation}

\subsection{Stepsizes}
The primal and dual stepsizes are chosen as in standard IPMs, which ensure that the updated iterate remains in the non-negative orthant after transitioning from the current iterate $x^k$, $z^k$, $w^k$, and $s^k$.
Given the search directions $\Delta x^k$, $\Delta z^k$, $\Delta w^k$, and $\Delta s^k$, denote the sets of indices for negative components by
\begin{align*}
I_{x} &= \{i \in \mathcal{I} \cup \mathcal{J} \mid \Delta x^k_{i} < 0 \}; \quad I_{w} = \{i \in \mathcal{J} \mid \Delta w^k_{i} < 0 \}; \\
I_{z} &= \{i \in \mathcal{I} \cup \mathcal{J} \mid \Delta z^k_{i} < 0 \}; \quad I_{s} = \{i \in \mathcal{J} \mid \Delta s^k_{i} < 0 \}.
\end{align*}
The primal and dual stepsizes are defined as follows: 
\begin{align}   
\alpha_{p} &= 0.995 \min\left\{1, \min_{i \in I_x} \left\{- \frac{x^k_{i}}{\Delta x^k_{i}}\right\} , \min_{i \in I_w} \left\{- \frac{w^k_{i}}{\Delta w^k_{i}}\right\} \right\}, \label{eq:stepsize-p} \\
\alpha_{d} &= 0.995 \min\left\{1, \min_{i \in I_z} \left\{- \frac{z^k_{i}}{\Delta z^k_{i}}\right\} , \min_{i \in I_s} \left\{- \frac{s^k_{i}}{\Delta s^k_{i}}\right\} \right\}. \label{eq:stepsize-d}
\end{align}
\section{Experimental Details}

\subsection{QP formulation for portfolio optimization} \label{sec:factor-model}

Consider a portfolio optimization problem, which aims at determining the asset allocation to maximize risk-adjusted returns while constraining correlation with market indexes or competing portfolios:
\begin{equation} \label{eq:portfolio-supp}
\begin{array}{cl}
   \minf  & \displaystyle - r^T x + \gamma x^T \Sigma x \\
   \st  & M x \leq u, ~ \mathbbm{1}_n^T x = 1, ~ x \geq 0,
\end{array}
\end{equation}
where variable $x \in \mathbb{R}^n$ represents the portfolio, $r \in \mathbb{R}^n$ denotes the vector of expected returns, $\gamma > 0$ denotes the risk aversion parameter, $\Sigma \in \mathbb{S}_n^{+}(\mathbb{R})$ represents the risk model covariance matrix, each row of $M \in \mathbb{R}^{d \times n}$ represents another portfolio, and $u \in \mathbb{R}^{d}$ upper bounds the correlations.

We assume a factor model for the covariance matrix $\Sigma = F F^T + D$, where $F \in \R^{n \times s}$ is the factor loading matrix and $D \in \R^{n \times n}$ is a diagonal matrix representing asset-specific risk. 
By replacing $\Sigma$ with $F F^T + D$ in \eqref{eq:portfolio-supp} and introducing a new variable $y = F^T x$, we write an equivalent problem in variables $x$ and $y$:
\begin{equation} \label{eq:portfolio-risk}
\begin{array}{cl}
   \minf  & \displaystyle - \gamma^{-1} r^T x + x^T D x + y^T y \\
   \st  & y = F^T x, ~ M x \leq u, ~ \mathbbm{1}_n^T x = 1, ~ x \geq 0, ~ y: \text{free}.
\end{array}
\end{equation}
Problem \eqref{eq:portfolio-risk} can be transformed into form $(\tilde{\text{P}})$ via standard techniques.

\subsection{Support vector machine (SVM) formulations in QP} \label{sec:SVM-formulation}

The linear support vector machine (SVM) problem solves a binary classification task on $n$ samples with $d$ features \cite[Chapter 12]{deisenroth2020mathematics}. 
Let $X \in \R^{d \times n}$ be a feature matrix whose columns are the attribute vectors $x_i \in \R^{d}$ associated with the $i$-th sample, $i = 1, \ldots, n$, and let $y_i \in \{-1, 1\}$ be the corresponding binary classification label.
The dual linear SVM with $\ell_1$-regularization can be formulated as a convex quadratic program \cite{fine2001efficient, woodsend2011exploiting, gondzio2009exploiting, vciegis2009high, woodsend2009hybrid}:
\begin{equation} \label{eq:dual-SVM}
\begin{array}{cl}
   \minf & \displaystyle\frac{1}{2} v^T v - \sum_{i=1}^n p_i \\
   \st  & v - X \diag(y) p = 0, ~ y^Tp = 0, \\
    & v \text{: free}, ~ 0 \leq p_i \leq \tau, ~ i = 1, \ldots, n,
\end{array}
\end{equation} 
where $v \in \R^{d}$ and $p \in \R^n$ are optimization variables, and $\tau > 0$ is the penalty parameter for misclassification.
The dual SVM problem \eqref{eq:dual-SVM} can be formulated into form $(\tilde{\text{P}})$ by setting
\begin{gather*}
x = \begin{bmatrix}
v \\ p
\end{bmatrix} \in \R^{d + n}, ~ 
Q = \begin{bmatrix}
I_d & \mathbf{0} \\ \mathbf{0} & \mathbf{0}
\end{bmatrix} \in \R^{(d+n) \times (d+n)}, \\ 
A = \begin{bmatrix}
I_d & -X \diag(y) \\
\mathbf{0} & y^T
\end{bmatrix} \in \R^{(d + 1) \times (d+n)}, ~
c = \begin{bmatrix}
\mathbf{0}_d \\ \mathbbm{1}_n
\end{bmatrix} \in \R^{d + n}, ~
b = \mathbf{0}_{d+1}, \\
\F = \{1, 2, \ldots, d\}, ~\I = \emptyset , ~ \J = \{d+1, \ldots, d+n\}, ~u^{\J} = \tau \mathbbm{1}_{|\J|}.
\end{gather*}
Note that the constraint matrix $A$ has a dense $(1, 2)$-block if the feature matrix $X$ is dense. 

\subsection{Regularization parameters in the experiments} \label{sec:reg-params}

\Cref{tab:reg_table_1} and \Cref{tab:reg_table_2} present the primal and dual regularization parameters $\rho_k$ and $\delta_k$ for Nys-IP-PMM on problems in \cref{sec:numerical-exp}.

\begin{table}[H]
   \centering
   \resizebox{\linewidth}{!}{
   \begingroup
   \begin{tabular}{ccccccccccc}
   \toprule
    & \multicolumn{2}{c}{\begin{tabular}[c]{@{}c@{}}\textbf{Portfolio} \\ \textbf{(section 5.1)}\end{tabular}} & \multicolumn{2}{c}{\begin{tabular}[c]{@{}c@{}}\textbf{CIFAR10\_1000} \\ \textbf{(section 5.2.2)}\end{tabular}} & \multicolumn{2}{c}{\textbf{RNASeq}} & \multicolumn{2}{c}{\textbf{SensIT}} & \multicolumn{2}{c}{\textbf{sector}} \\
   \cmidrule(lr){2-3} \cmidrule(lr){4-5} \cmidrule(lr){6-7} \cmidrule(lr){8-9} \cmidrule(lr){10-11} 
   \(k\) & \(\rho_k\) & \(\delta_k\) & \(\rho_k\) & \(\delta_k\) & \(\rho_k\) & \(\delta_k\) & \(\rho_k\) & \(\delta_k\) & \(\rho_k\) & \(\delta_k\) \\
   \midrule
       0 & \num{8.00e+00} & \num{8.00e+00} & \num{8.00e+00} & \num{8.00e+00} & \num{8.00e+00} & \num{8.00e+00} & \num{8.00e+00} & \num{8.00e+00} & \num{8.00e+00} & \num{8.00e+00} \\
       1 & \num{7.12e-01} & \num{7.12e-01} & \num{3.81e+00} & \num{1.71e+00} & \num{2.70e+00} & \num{4.03e-02} & \num{3.56e+00} & \num{1.33e+00} & \num{2.96e+00} & \num{2.96e+00} \\
       2 & \num{1.02e-01} & \num{1.02e-01} & \num{1.82e+00} & \num{3.67e-01} & \num{2.64e-02} & \num{3.95e-04} & \num{8.03e-01} & \num{3.00e-01} & \num{1.13e+00} & \num{2.15e-01} \\
       3 & \num{4.77e-02} & \num{2.03e-02} & \num{1.24e+00} & \num{2.50e-01} & \num{5.21e-03} & \num{7.77e-05} & \num{2.42e-01} & \num{9.05e-02} & \num{8.98e-01} & \num{1.48e-01} \\
       4 & \num{2.37e-02} & \num{4.95e-03} & \num{8.19e-01} & \num{1.65e-01} & \num{2.24e-03} & \num{3.34e-05} & \num{1.53e-01} & \num{5.73e-02} & \num{7.14e-01} & \num{1.02e-01} \\
       5 & \num{3.40e-03} & \num{7.10e-04} & \num{5.87e-01} & \num{1.19e-01} & \num{1.63e-03} & \num{1.98e-05} & \num{8.99e-02} & \num{3.36e-02} & \num{6.07e-01} & \num{8.71e-02} \\
       6 & \num{1.91e-04} & \num{3.99e-05} & \num{1.22e-01} & \num{2.45e-02} & \num{1.39e-03} & \num{1.68e-05} & \num{4.95e-02} & \num{1.85e-02} & \num{3.27e-01} & \num{4.70e-02} \\
       7 & \num{2.86e-05} & \num{5.97e-06} & \num{5.66e-02} & \num{1.14e-02} & \num{1.16e-03} & \num{1.40e-05} & \num{2.90e-02} & \num{1.08e-02} & \num{2.18e-01} & \num{3.13e-02} \\
       8 & \num{5.13e-06} & \num{1.07e-06} & \num{1.31e-02} & \num{2.64e-03} & \num{1.13e-03} & \num{1.35e-05} & \num{1.59e-02} & \num{5.94e-03} & \num{1.97e-01} & \num{2.82e-02} \\
       9 & \num{2.17e-06} & \num{1.42e-07} & \num{2.84e-03} & \num{5.74e-04} & \num{1.10e-03} & \num{1.29e-05} & \num{9.83e-03} & \num{3.67e-03} & \num{1.21e-01} & \num{1.74e-02} \\
       10 & \num{1.57e-06} & \num{1.03e-07} & \num{4.72e-04} & \num{9.53e-05} & \num{1.08e-03} & \num{1.28e-05} & \num{5.04e-03} & \num{1.88e-03} & \num{8.03e-02} & \num{1.15e-02} \\
       11 & \num{3.13e-07} & \num{2.05e-08} & \num{5.39e-06} & \num{1.09e-06} & \num{7.86e-04} & \num{9.28e-06} & \num{2.82e-03} & \num{1.05e-03} & \num{5.09e-02} & \num{7.31e-03} \\
       12 & \num{1.87e-08} & \num{1.22e-09} & \num{4.32e-08} & \num{8.73e-09} & \num{5.76e-04} & \num{6.80e-06} & \num{1.45e-03} & \num{5.43e-04} & \num{2.18e-02} & \num{3.13e-03} \\
       13 & \num{5.00e-10} & \num{5.00e-10} & \num{5.00e-10} & \num{2.95e-09} & \num{3.31e-04} & \num{3.91e-06} & \num{6.73e-04} & \num{2.52e-04} & \num{9.95e-03} & \num{1.43e-03} \\
       14 & \num{5.00e-10} & \num{5.00e-10} & \num{5.00e-10} & \num{5.00e-10} & \num{5.95e-05} & \num{7.02e-07} & \num{3.26e-04} & \num{1.22e-04} & \num{5.16e-03} & \num{7.40e-04} \\
       15 & \num{5.00e-10} & \num{5.00e-10} & \num{5.00e-10} & \num{5.00e-10} & -- & --  & \num{1.55e-04} & \num{5.78e-05} & \num{1.79e-03} & \num{2.56e-04} \\
       16 & \num{5.00e-10} & \num{5.00e-10} & -- & --  & -- & --  & -- & --  & \num{3.22e-04} & \num{4.62e-05} \\
       17 & \num{5.00e-10} & \num{5.00e-10} & -- & --  & -- & --  & -- & --  & -- & --  \\
   \bottomrule
   \end{tabular}
   \endgroup
   }
   \caption{Regularization parameters for for experiments in \cref{sec:numerical-exp}.}
   \label{tab:reg_table_1}
\end{table}

\begin{table}[H]
   \centering
   \resizebox{\linewidth}{!}{
   \begingroup
   \begin{tabular}{ccccccccc}
   \toprule
    & \multicolumn{2}{c}{\textbf{CIFAR10}} & \multicolumn{2}{c}{\textbf{STL10}} & \multicolumn{2}{c}{\textbf{arcene}} & \multicolumn{2}{c}{\textbf{dexter}} \\
   \cmidrule(lr){2-3} \cmidrule(lr){4-5} \cmidrule(lr){6-7} \cmidrule(lr){8-9} 
   \(k\) & \(\rho_k\) & \(\delta_k\) & \(\rho_k\) & \(\delta_k\) & \(\rho_k\) & \(\delta_k\) & \(\rho_k\) & \(\delta_k\) \\
   \midrule
       0 & \num{8.00e+00} & \num{8.00e+00} & \num{8.00e+00} & \num{8.00e+00} & \num{8.00e+00} & \num{8.00e+00} & \num{8.00e+00} & \num{8.00e+00} \\
       1 & \num{4.18e+00} & \num{2.27e+00} & \num{2.72e+00} & \num{7.29e-02} & \num{2.85e+00} & \num{2.66e-01} & \num{2.85e+00} & \num{2.66e-01} \\
       2 & \num{7.35e-01} & \num{3.99e-01} & \num{2.18e-02} & \num{5.83e-04} & \num{1.42e-02} & \num{1.33e-03} & \num{1.43e-02} & \num{1.33e-03} \\
       3 & \num{4.01e-01} & \num{2.17e-01} & \num{4.29e-03} & \num{1.15e-04} & \num{8.07e-05} & \num{7.53e-06} & \num{1.28e-04} & \num{1.19e-05} \\
       4 & \num{2.41e-01} & \num{1.31e-01} & \num{2.37e-03} & \num{6.34e-05} & \num{3.54e-06} & \num{3.31e-07} & \num{1.51e-05} & \num{1.41e-06} \\
       5 & \num{1.02e-01} & \num{5.55e-02} & \num{1.40e-03} & \num{2.43e-05} & \num{7.72e-07} & \num{7.20e-08} & -- & --  \\
       6 & \num{4.51e-02} & \num{2.44e-02} & \num{6.85e-04} & \num{1.19e-05} & -- & --  & -- & --  \\
       7 & \num{2.24e-02} & \num{1.22e-02} & \num{2.88e-04} & \num{5.02e-06} & -- & --  & -- & --  \\
       8 & \num{9.58e-03} & \num{5.20e-03} & \num{1.10e-04} & \num{1.92e-06} & -- & --  & -- & --  \\
       9 & \num{3.55e-03} & \num{1.93e-03} & \num{3.69e-05} & \num{6.43e-07} & -- & --  & -- & --  \\
       10 & \num{1.28e-03} & \num{6.92e-04} & \num{3.55e-06} & \num{6.20e-08} & -- & --  & -- & --  \\
       11 & \num{2.53e-04} & \num{1.37e-04} & \num{2.54e-07} & \num{4.43e-09} & -- & --  & -- & --  \\
       12 & \num{3.48e-05} & \num{1.89e-05} & \num{7.18e-09} & \num{5.00e-10} & -- & --  & -- & --  \\
   \bottomrule
   \end{tabular}
   \endgroup
   }
   \caption{Regularization parameters for for experiments in \cref{sec:numerical-exp} (continued).}
   \label{tab:reg_table_2}
\end{table}

\subsection{Spectrums of $AA^T$ and $N_k$}

\Cref{fig:eigvals-dist} illustrates the eigenvalue distributions of $AA^T$ and $N_k$ for SVM problem in \cref{sec:exp-cond}.
Our Nys-IP-PMM converges at $k=12$ for a relative tolerance $\epsilon = 10^{-6}$; and at $k=15$ for $\epsilon = 10^{-8}$.
As the algorithm progresses, the normal equation matrix $N_k$ becomes increasingly ill-conditioned, which is typical for IPMs.
For this problem, the rapid decay of the top eigenvalues of $N_k$ suggests the suitability and success of the Nyström preconditioner.

\begin{figure}
   \centering
   \includegraphics[width=\linewidth]{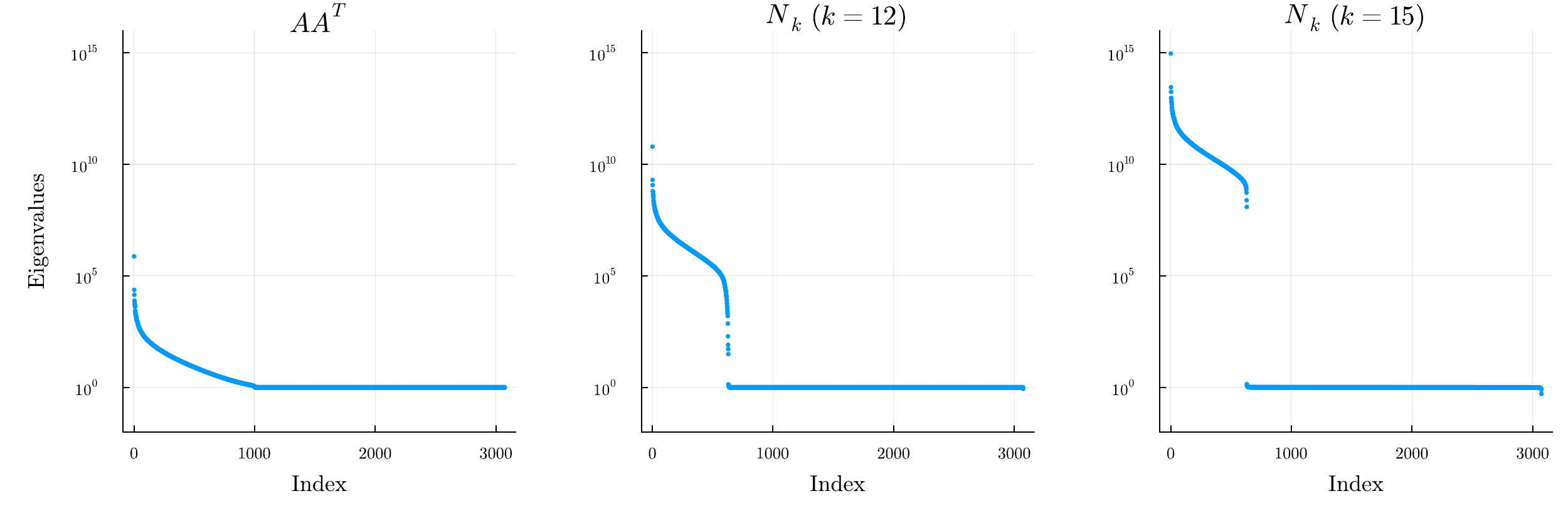}
   \caption{
   Distribution of eigenvalues for $AA^T$ and $N_k$ at different IP-PMM iterations of the SVM problem formed from 1000 samples of CIFAR10.}
   \label{fig:eigvals-dist}
\end{figure}

\end{document}